\documentclass[12pt]{amsart}
 
\pdfoutput=1
\usepackage[margin=1in]{geometry}
\usepackage{amsmath,amsthm,amssymb,amsrefs,bbm,color,esint,esvect,float,graphicx,mathrsfs}
\usepackage[bookmarksnumbered, colorlinks, plainpages,linkcolor=blue,anchorcolor=blue,citecolor=blue,urlcolor=blue]{hyperref}
\usepackage{todonotes}

\newcommand{\mc}{\mathcal}

\DeclareMathOperator{\ind}{\chi}

\usepackage{euscript,latexsym}
\usepackage{accents}

\usepackage{epstopdf}
\AppendGraphicsExtensions{.tif}
\newcommand{\N}{\mathbb{N}}

\newcommand{\C}{\mathbb{C}} 

    \newcommand{\DD}{\mathcal{D}}

 \newcommand{\vertiii}[1]{{\left\vert\kern-0.25ex\left\vert\kern-0.25ex\left\vert #1 
    \right\vert\kern-0.25ex\right\vert\kern-0.25ex\right\vert}}

 \newtheorem{thm}{Theorem}[section]
 \newtheorem{lemma}[thm]{Lemma}
 \newtheorem{cor}[thm]{Corollary}
 \newtheorem{prop}[thm]{Proposition}
 \newtheorem{rem}[thm]{Remark}
 
 \newtheorem{conj}[thm]{Conjecture}
 
 \numberwithin{equation}{section}

\theoremstyle{definition}

\begin{document}

\title[Weighted weak-type bounds for the Bergman projection]{Weak-type bounds for the Bergman projection with Bekoll\'e-Bonami weights}

\author{Jiale Chen}
\thanks{J. C. is supported by National Natural Science Foundation of China (No. 12501170).}
\address{Jiale Chen, School of Mathematics and Statistics, Shaanxi Normal University, Xi'an 710119, China}
\email{jialechen@snnu.edu.cn}

\author{Zoe Nieraeth}
\thanks{Z. N. is supported by the Basque government through project GV IT1615-22}
\address{Zoe Nieraeth (she/her), University of the Basque country (UPV/EHU), Leioa, Spain}
\email{zoe.nieraeth@gmail.com}

\author{Cody B. Stockdale}
\address{Cody B. Stockdale, School of Mathematical Sciences and Statistics, Clemson University, Clemson, SC 29634, USA}
\email{cbstock@clemson.edu}

\author{Nathan A. Wagner}
\thanks{N. A. W. was supported by the National Science Foundation, DMS MSPRF Grant No. 2203272 and is currently supported by the National Science Foundation, DMS Grant No. 2549719}
\address{Nathan A. Wagner, Department of Mathematical Sciences, George Mason University, Fairfax, VA 22030, USA}
\email{nwagner8@gmu.edu}

\begin{abstract}
We establish weighted weak-type bounds for the Bergman projection with respect to Bekoll\'e-Bonami characteristics. We present two proofs of an improved quantitative weak-type $(1,1)$ estimate, as well as sharp weak-type $(p,p)$ bounds for $p>1$ and mixed weighted weak-type $(1,1)$ inequalities. Our results, which hold for a wide class of simple domains in $\mathbb{C}^n$, are new even in the classical settings of the upper half-plane and the unit disk.\looseness=-1
\end{abstract}



\maketitle


\section{Introduction}\label{IntroductionSection}
\allowdisplaybreaks[4]

The Bergman projection $P$ of a domain $\Omega \subseteq \mathbb{C}^n$ is the orthogonal projection from $L^2(\Omega)$ onto the Bergman space 
$$
    \mathcal{A}^2(\Omega) := L^2(\Omega)\cap \text{Hol}(\Omega),
$$ 
where $L^2(\Omega)$ is the space of square-integrable functions on $\Omega\subseteq \mathbb{C}^n$ with respect to the Lebesgue measure on $\mathbb{C}^n=\mathbb{R}^{2n}$ and $\text{Hol}(\Omega)$ is the space of holomorphic functions on $\Omega$. Although, as we will see, our bounds hold for a general class of simple domains $\Omega\subseteq \mathbb{C}^n$, we focus on the case when $\Omega$ is the upper half-plane $\mathbb{H} := \{z \in \mathbb{C} \colon \text{Im}(z)>0\}$ until Section \ref{SimpleDomainsSection} for concreteness and simplicity. In this situation, $P$ can be expressed as 
$$
    Pf(z) = \frac{1}{\pi}\int_{\mathbb{H}} \frac{f(w)}{(z-\overline{w})^2}\,dA(w), 
$$
where $A$ denotes the Lebesgue area measure. 

As $P$ is a projection, it is automatically bounded on $L^2(\mathbb{H})$. More generally, $P$ acts boundedly on $L^p(\mathbb{H})$ for all $p \in (1,\infty)$. While the strong-type bound of $P$ on $L^p(\mathbb{H})$ fails at the endpoint $p=1$, the weak-type bound from $L^1(\mathbb{H})$ to $L^{1,\infty}(\mathbb{H})$ instead holds -- we have for general $p \in [1,\infty)$ that
$$
    \|Pf\|_{L^{p,\infty}(\mathbb{H})}:= \sup_{\lambda>0}\lambda A(\{z \in \mathbb{H}\colon |Pf(z)|>\lambda\})^{\frac{1}{p}}\lesssim  \|f\|_{L^p(\mathbb{H})}
$$
for all $f \in L^p(\mathbb{H})$. In \cites{BB1978,B198182}, Bekoll\'e and Bonami generalized the $L^p$ theory of the Bergman projection to weighted spaces -- we have for any $p \in (1,\infty)$ and weight $\sigma$ that $P$ is bounded on $L^p(\mathbb{H},\sigma)$ if and only if $\sigma \in B_p$, i.e.,  
$$
    [\sigma]_{B_p}:= \sup_{I} \langle \sigma \rangle_{Q_I}\langle \sigma^{1-p'}\rangle_{Q_I}^{p-1} < \infty,
$$
where the supremum is taken over all intervals $I$ on the real axis, $Q_I:=I \times (0,\ell(I)]$ is the Carleson tent over $I$, $\langle \sigma \rangle_{Q_I} := \frac{1}{A(Q_I)} \int_{Q_I} |\sigma|\,dA$, and $\frac{1}{p}+\frac{1}{p'}=1$. The $B_p$ condition also characterizes the weighted weak-type bounds for $P$ as follows: if $p \in [1,\infty)$ and $\sigma$ is a weight, then $P$ is bounded from $L^p(\mathbb{H},\sigma)$ to $L^{p,\infty}(\mathbb{H},\sigma)$ if and only if $\sigma \in B_p$, where $\sigma \in B_1$ means \looseness=-1
$$
    [\sigma]_{B_1} := \sup_{I} \langle \sigma \rangle_{Q_I}\|\sigma^{-1}\|_{L^{\infty}(Q_I)} < \infty.
$$
The qualitative aspects of the weighted theory for $P$ have been extensively studied, and the quantitative strong-type bounds in terms of $[\sigma]_{B_p}$ are also well-understood. In particular, Pott and Reguera proved the following sharp bound in \cite{PR2013}: if $p \in (1,\infty)$ and $\sigma \in B_p$, then 
\begin{align}\label{PStrongType}
    \|P\|_{L^p(\mathbb{H},\sigma)\rightarrow L^{p}(\mathbb{H},\sigma)} \lesssim [\sigma]_{B_p}^{\max\big(1,\frac{p'}{p}\big)}.
\end{align}
We note that \eqref{PStrongType} also holds with $P$ replaced by the positive Bergman operator
$$
    P^+f(z) = \frac{1}{\pi}\int_{\mathbb{H}} \frac{|f(w)|}{|z-\overline{w}|^2}\,dA(w).
$$
See \cites{APR2019, RTW2017, HWW20201, HWW20202, SW2022, SW2023, WW2021} for more on the weighted $L^p$ theory of the Bergman projection. 

We establish quantitative weak-type bounds for $P$ and $P^+$ with respect to $[\sigma]_{B_p}$. Our first result concerns the $p=1$ endpoint case and addresses \cite{SW2023}*{Open Question 1.16}. Below, we set \looseness=-1
$$
    [\sigma]_{B_{\infty}} := \sup_{I} \frac{1}{\sigma(Q_I)}\int_{Q_I} M(\sigma \chi_{Q_I})(z)\,dA(z),
$$
where $M$ denotes the Carleson tent maximal operator given by 
$$
    Mf := \sup_{I} \langle |f|\rangle_{Q_I}\chi_{Q_I}.
$$
Recall that $[\sigma]_{B_{\infty}} \lesssim [\sigma]_{B_p}$ and that the condition $[\sigma]_{B_{\infty}}< \infty$ characterizes the membership of $\sigma$ in the class $B_{\infty}$, which properly contains $\bigcup_{p\ge 1} B_p$; see \cites{APR2019, MP2025}. 
\begin{thm}\label{WeightedWeakL1}
If $\sigma \in B_1$, then
$$
    \|P^+\|_{L^1(\mathbb{H},\sigma)\rightarrow L^{1,\infty}(\mathbb{H},\sigma)}\lesssim [\sigma]_{B_1}(\log [\sigma]_{B_{\infty}} + 1).
$$
\end{thm}

We conjecture that the bound in Theorem \ref{WeightedWeakL1} is sharp with respect to $[\sigma]_{B_1}$. In analogy, for Calder\'on-Zygmund singular integral operators $T$ and Muckenhoupt $A_1$ weights $w$ on $\mathbb{R}^n$, we have the following estimate from \cites{LOP2009, HP2013}: if $w \in A_1$, then 
$$
	\|T\|_{L^1(\mathbb{R}^n,w)\rightarrow L^{1,\infty}(\mathbb{R}^n,w)} \lesssim [w]_{A_1}(\log[w]_{A_{\infty}} +1),
$$
and this bound is the best possible in terms of $[w]_{A_1}$ when $T$ is the Hilbert transform; see \cite{LNO2020}.\looseness=-1

Aside from the current work, the only proofs of the weighted weak-type $(1,1)$ bound of $P$ for general $\sigma \in B_1$ exist in \cites{B198182, SW2022, GW2024} -- the argument handling the generality of \cite{SW2022} only allows for a qualitative bound, while optimized adaptations of the proofs in \cites{B198182, GW2024} yield 
\begin{align}\label{B1SquareBound}
    \|P\|_{L^1(\mathbb{H},\sigma)\rightarrow L^{1,\infty}(\mathbb{H},\sigma)}\lesssim [\sigma]_{B_1}^2.
\end{align}
This dependence was improved for the sub-class of $B_1$ weights of bounded hyperbolic oscillation introduced in \cite{APR2019}. For an interval $I$, we define $T_I := I \times \big(\frac{\ell(I)}{2},\ell(I)\big]$ and say that a weight $\sigma$ is of bounded hyperbolic oscillation if there exists $c_{\sigma} > 0$ such that
$$
    \sigma(z) \leq c_{\sigma} \sigma(w) 
$$
for all intervals $I \subseteq \mathbb{R}$ and all $z,w \in T_I$. The third and fourth authors proved in \cite{SW2023}*{Remark 1.11} that if $\sigma \in B_1$ is of bounded hyperbolic oscillation, then 
\begin{align}\label{APRWeakL1}
    \|P^+\|_{L^1(\mathbb{H},\sigma)\rightarrow L^{1,\infty}(\mathbb{H},\sigma)}\lesssim c_{\sigma}^{\frac{1}{8 [\sigma]_{B_\infty}}}[\sigma]_{B_1}(\log[\sigma]_{B_{\infty}}+1).
\end{align}
Theorem \ref{WeightedWeakL1} quantitatively improves \eqref{B1SquareBound} and qualitatively generalizes \eqref{APRWeakL1} to all $\sigma \in B_1$.

The bounded hyperbolic oscillation assumption ensures the following reverse-H\"older inequality: if $\sigma \in B_{\infty}$ is of bounded hyperbolic oscillation and $r\in \big(1,1+\frac{1}{8[\sigma]_{B_{\infty}}}\big)$, then 
\begin{equation}\label{ReverseHolder}
 \langle \sigma^r \rangle_{Q_I}^{\frac{1}{r}} \lesssim c_{\sigma}^{\frac{1}{r'}}\langle \sigma \rangle_{Q_I}
\end{equation}
for all intervals $I \subseteq \mathbb{R}$; see \cite{SW2023}*{Theorem 2.9}. Such a property, which does not hold for general $B_1$ weights, is crucial to the existing arguments for establishing sharp weighted weak-type $(1,1)$ bounds for singular integrals from \cites{LOP2009, HP2013, DsLR2016, FN2019}. As such, the main obstacle in proving our results is this lack of reverse-H\"older inequality -- we overcome this difficulty with new ideas from dyadic harmonic analysis. 

We present two proofs of Theorem \ref{WeightedWeakL1}. Our first proof, which is motivated by the argument of the second and third authors for operators satisfying sparse form domination from \cite{NS2024}, has the feature that it extends to all $p \in [1,\infty)$ and yields 
\begin{align}\label{WeakLpWithLog}
    \|P^+\|_{L^p(\mathbb{H},\sigma)\rightarrow L^{p,\infty}(\mathbb{H},\sigma)} \lesssim [\sigma]_{B_p}^{\frac{1}{p}}[\sigma]_{B_{\infty}}^{\frac{1}{p'}}(\log [\sigma]_{B_{\infty}}+1);
\end{align}
see Remark \textcolor{blue}{\ref{p>1Remark}.} We show that the logarithmic factor in \eqref{WeakLpWithLog} can be removed when $p>1$ and obtain the following sharp bound.
\begin{thm}\label{WeightedWeakLp}
If $p \in (1,\infty)$ and $\sigma \in B_p$, then  
$$
\|P^+\|_{L^p(\mathbb{H},\sigma)\rightarrow L^{p,\infty}(\mathbb{H},\sigma)} \lesssim [\sigma]_{B_p}^{\frac{1}{p}}[\sigma]_{B_{\infty}}^{\frac{1}{p'}} \lesssim [\sigma]_{B_p}.
$$
Moreover, this bound is sharp, even for the classical projection operator $P$, in the sense that if $\phi$ is an increasing function such that $\|P\|_{L^p(\mathbb{H},\sigma)\rightarrow L^{p,\infty}(\mathbb{H},\sigma)} \leq \phi([\sigma]_{B_p})$ for all $\sigma \in B_p$, then $\phi(t) \gtrsim t$ for all large enough $t$. 
\end{thm}
\noindent Before \eqref{WeakLpWithLog} and Theorem \ref{WeightedWeakLp}, the previously best known bound for $\|P^+\|_{L^p(\Omega,\sigma)\rightarrow L^{p,\infty}(\Omega,\sigma)}$ for $p>1$ was the one inherited from the sharp strong-type estimate \eqref{PStrongType}. We prove the upper bound in Theorem \ref{WeightedWeakLp} using testing conditions inspired by the arguments from \cites{LSU2009, HL2018}, and we deduce its optimality using a modification of the example from \cite{PR2013}*{Section 5}.

We obtain the following strong-type bound as a consequence of Theorem \ref{WeightedWeakLp} and the Marcinkiewicz interpolation theorem.
\begin{cor}\label{StrongPCorollary}
If $1\leq q<p<\infty$ and $\sigma\in B_q$, then
$$
    \|P^+\|_{L^p(\mathbb{H},\sigma)\to L^p(\mathbb{H},\sigma)}\lesssim [\sigma]_{B_q}^{\frac{1}{p}}[\sigma]_{B_{\infty}}^{\frac{1}{p'}}.
$$
\end{cor}
\noindent Corollary \ref{StrongPCorollary} is a refined Bergman space version of the following analogous bound for Calder\'{o}n-Zygmund operators with respect to $A_q$ weights obtained in \cite{HLMORSUt2012}*{Corollary 12.1}:
if $T$ is a Calder\'on-Zygmund operator, $1\leq q<p<\infty$, and $w \in A_q$, then
$$
    \|T\|_{L^p(\mathbb{R}^n,w)\rightarrow L^p(\mathbb{R}^n,w)} \lesssim [w]_{A_q}.
$$
This bound was conjectured in \cite{LO2012}*{Conjecture 1.3} and earlier progress was made in \cites{LOP2008, LOP2009}.

Our second proof of Theorem \ref{WeightedWeakL1} utilizes a different set of ideas inspired by \cite{CRr2020} that lead to the following mixed weighted weak-type $(1,1)$ bound. 
\begin{thm}\label{MixedWeak}
If either 
\begin{enumerate}
\addtolength{\itemsep}{0.1cm}
\item $u \in B_1$ and $v \in B_p(u)$ for $1<p<\infty$ are such that $uv$ is of bounded hyperbolic oscillation, or
\item $v \in B_1$ is of bounded hyperbolic oscillation and $u \in B_1(v)$,
\end{enumerate}
then 
$$
	\|v^{-1}P^+(fv)\|_{L^{1,\infty}(\mathbb{H},uv)} \lesssim C_{u,v}\|f\|_{L^1(\mathbb{H},uv)}
$$
for all $f \in L^1(\mathbb{H},uv)$, where, in the first case,
$$
  C_{u,v}:=(\log c_{uv}+1)[uv]_{B_{\infty}}[u]_{B_{1}}
     \left(\log([uv]_{B_{\infty}}[v]_{B_{p}(u)}[u]_{B_{1}}) + 1\right),
$$
and, in the second case,
$$
C_{u,v}:=(\log c_v+1)[v]_{B_{1}}[v]_{B_{\infty}}[u]_{B_{1}(v)}
  \big(\log([v]_{B_{1}}[uv]_{B_{\infty}})+1\big).
$$
\end{thm}
\noindent Above, the weighted $B_p(u)$ characteristics are defined in the same way as the $B_p$ constants, but with the Lebesgue measure replaced by the absolutely continuous measure with density $u$. \looseness=-1

The second case of Theorem \ref{MixedWeak} reduces to Theorem \ref{WeightedWeakL1} when $v \equiv 1$. Additionally, when $u\equiv1$, the first case of Theorem \ref{MixedWeak} gives the following Bergman projection analogue of \cite{OPR2016}*{Theorem 1.17}; see also \cite{S2020}*{Theorem 2}.
\begin{cor}
    If $p \in [1,\infty)$ and $v\in B_{p}$ is of bounded hyperbolic oscillation, then
$$
    \left\|v^{-1}P^+(fv)\right\|_{L^{1,\infty}(\mathbb{H},v)}
\lesssim(\log c_v+1)[v]_{B_{\infty}}(\log[v]_{B_{p}}+1)\|f\|_{L^1(\mathbb{H},v)}
$$
for all $f \in L^1(\mathbb{H},v)$. 
\end{cor}

Theorem \ref{MixedWeak} is a special case of the Bergman space version of Sawyer's conjecture, which states that if $T$ is a Calder\'on-Zygmund operator, $u \in A_1$, and $v \in A_{\infty}$, then 
$$
	\|v^{-1}T(fv)\|_{L^{1,\infty}(\mathbb{R}^n,uv)} \lesssim_{u,v} \|f\|_{L^1(\mathbb{R}^n,uv)}
$$
for all $f \in L^1(\mathbb{R}^n,uv)$. This estimate was proved in the case $u \in A_1$ and either $v \in A_1$ or $v \in A_{\infty}(u)$ in \cite{CuMP2005}, and the full result for general $u \in A_1$ and $v \in A_{\infty}$ was established in \cites{LOP2019}.  In analogy, we pose the following Bergman space version of Sawyer's conjecture. 
\begin{conj}\label{SawyersConjecture}
If $u \in B_1$ and $v \in B_{\infty}$, then there exists $C_{u,v}>0$ such that
$$
	\|v^{-1}P^+(fv)\|_{L^{1,\infty}(\mathbb{H},uv)} \lesssim C_{u,v} \|f\|_{L^1(\mathbb{H},uv)}
$$
for all $f \in L^1(\mathbb{H},uv)$, where $C_{u,v}$ depends on $[u]_{B_1}$ and $[v]_{B_{\infty}}$.
\end{conj}

As mentioned earlier, all of our bounds hold for any simple domain $\Omega \subseteq \mathbb{C}^n$ -- we call a pseudoconvex domain with smooth boundary simple if it is either \looseness=-1
\begin{enumerate}
\addtolength{\itemsep}{0.1cm}
    \item a bounded, finite-type domain in $\mathbb{C}^2$, 
    \item a bounded, decoupled, finite-type domain in $\mathbb{C}^n$,
    \item a bounded, convex, finite-type domain in $\mathbb{C}^n$, 
    \item a bounded, strongly pseudoconvex domain 
    in $\mathbb{C}^n$, or
    \item the upper half-plane $\mathbb{H}$ in $\C.$
\end{enumerate}
Notice that the unit disk and, more generally, the unit ball in $\mathbb{C}^n$ are simple domains. The reason for the lone addition of the upper half-plane is that the theory of unbounded pseudoconvex domains is much richer in higher dimensions, for example, in tube domains; see \cites{ BBPR, BBGNP}.
On the other hand, the upper half-plane is an ideal venue to present our ideas since its dyadic structure is geometrically cleaner than that of the unit disk, and our arguments in this setting adapt in a straightforward way to any bounded simple domain. 

The paper is organized as follows. In Section \ref{PreliminariesSection}, we define notation, describe the dyadic structure of $\mathbb{H}$, and collect preliminary results. In Section \ref{GeneralSection}, we prove our main results. We describe how our bounds extend to the generality of bounded simple domains in Section \ref{SimpleDomainsSection}.


\section{Preliminaries}\label{PreliminariesSection}

We write $A\lesssim B$ if $A\leq CB$ for some $C>0$ that may depend on underlying parameters other than weights, and write $A\approx B$ if $A\lesssim  B\lesssim A$. We call a locally integrable and almost everywhere positive function a weight. Given a weight $\sigma$ on $\mathbb{H}$ and $p>0$, the spaces $L^p(\mathbb{H},\sigma)$ and $L^{p,\infty}(\mathbb{H},\sigma)$ are defined to be the respective collections of all $f$ such that
$$
    \|f\|_{L^p(\mathbb{H},\sigma)}:=\bigg(\int_{\mathbb{H}}|f|^p\sigma\,dA\bigg)^{\frac{1}{p}} <\infty
$$
and
$$
    \|f\|_{L^{p,\infty}(\mathbb{H},\sigma)}:=\sup_{\lambda>0}\lambda \sigma(\{z \in \mathbb{H}: |f(z)|>\lambda\})^{\frac{1}{p}}<\infty.
$$
Given spaces $\mathcal{X}$ and $\mathcal{Y}$ and $T:\mathcal{X}\rightarrow \mathcal{Y}$, we write $\|T\|_{\mathcal{X}\rightarrow\mathcal{Y}}$ for the infimum of all $C>0$ so that\looseness=-1
$$
    \|Tf\|_{\mathcal{Y}} \leq C\|f\|_{\mathcal{X}}
$$
for all $f \in \mathcal{X}$. 

We call a collection of intervals $\mathcal{D}$ a dyadic grid if there exists $\alpha \in \{0,\frac{1}{3},\frac{2}{3}\}$ for which 
$$
    \mathcal{D} = \mathcal{D}^{\alpha} := \{2^{j}([0,1) + \alpha + k)\colon j,k \in \mathbb{Z}\}.
$$
Note that dyadic grids satisfy the property that if $I, I' \in \mathcal{D}$, then either $I \cap I' = \emptyset$, $I \subseteq I'$, or $I \supsetneq I'$. Moreover, we can reduce from continuous to dyadic situations using the fact that for any interval $I\subseteq \mathbb{R}$, there exists $\alpha\in\{0,\frac{1}{3},\frac{2}{3}\}$ and $J\in\mathcal{D}^\alpha$ such that $I\subseteq J$ and $\ell(J)\leq 6\ell(I)$. Recalling that $Q_I = I \times (0,\ell(I)]$ and $T_I = I \times \big(\frac{\ell(I)}{2},\ell(I)\big]$, we  have that $\{T_I\}_{I \in \mathcal{D}}$ forms a partition of $\mathbb{H}$ for any dyadic grid $\mathcal{D}$, and that  
$$
    A(Q_I) = 2A(T_I)
$$
for any interval $I \subseteq \mathbb{R}$. Observe that for $I \subseteq \mathcal{D}$, we have 
\begin{align}\label{SparseCondition}
    A\bigg(\bigcup_{\substack{I' \in \mathcal{D}\\ I' \subsetneq I}} Q_{I'}\bigg) = \sum_{\substack{I' \in \mathcal{D}\\ I' \subsetneq I}} A(T_{I'}) = \frac{1}{2}A(Q_I).
\end{align}

We heavily rely on the following domination of $P^+$. Below, for a dyadic grid $\mathcal{D}$, we set
$$
    S_{\mathcal{D}}f := \sum_{I \in \mathcal{D}} \langle f\rangle_{Q_I}\chi_{Q_I}.
$$
\begin{prop}[\cite{PR2013}*{Proposition 3.4}]\label{SparseBound}
    If $f$ is a locally integrable function on $\mathbb{H}$, then 
    $$
        P^+f(z) \lesssim \sum_{\alpha \in \{0, \frac{1}{3}, \frac{2}{3}\}} S_{\mathcal{D}^{\alpha}}(|f|)(z)
    $$
    for all $z \in \mathbb{H}$. 
\end{prop}

We will use the following structural property of $B_{\infty}$ weights. 
\begin{lemma}\label{cup}
If $\sigma\in B_{\infty}$, $\mathcal{D}$ is a dyadic grid, and $\mathcal{E}\subseteq\mathcal{D}$, then 
$$
    \sum_{I\in\mathcal{E}}\sigma(Q_I) \leq2[\sigma]_{B_{\infty}}\sigma\bigg(\bigcup_{I\in\mathcal{E}}Q_I\bigg).$$
\end{lemma}

\begin{proof} 
For each positive integer $k$, let
$$
    \mathcal{E}_k:=\{I\in\mathcal{E}\colon I\subseteq[-2^k,2^k) \,\, \text{and}\,\, \ell(I)>2^{-k}\}.
$$
Then $\mathcal{E}_k$ is a finite collection of dyadic intervals. Let $\mathcal{E}^*_k$ be the family of maximal intervals of $\mathcal{E}_k$ with respect to inclusion. Then for each $J\in\mathcal{E}^*_k$, we have 
\begin{align*}
\sum_{\substack{I\in\mathcal{E}_k\\ I\subseteq J}}\sigma(Q_I)
&=2\sum_{\substack{I\in\mathcal{E}_k\\ I\subseteq J}}\langle\sigma\rangle_{Q_I}A(T_I)\\
&\leq2\sum_{\substack{I\in\mathcal{E}_k\\ I\subseteq J}}\int_{T_I}M(\sigma\chi_{Q_J})\,dA\\
&\leq2\int_{Q_J}M(\sigma\chi_{Q_J})\,dA\\
&\leq2[\sigma]_{B_{\infty}}\sigma(Q_J),
\end{align*}
and, consequently, 
\begin{align*}
\sum_{I\in\mathcal{E}_k}\sigma(Q_I)
&=\sum_{J\in\mathcal{E}^*_k}\sum_{\substack{I\in\mathcal{E}_k\\ I\subseteq J}}\sigma(Q_I)\\
&\leq2[\sigma]_{B_{\infty}}\sum_{J\in\mathcal{E}^*_k}\sigma(Q_J)\\
&=2[\sigma]_{B_{\infty}}\sigma\bigg(\bigcup_{J\in\mathcal{E}^*_k}Q_J\bigg)\\
&\leq2[\sigma]_{B_{\infty}}\sigma\bigg(\bigcup_{I\in\mathcal{E}}Q_I\bigg).
\end{align*}
Letting $k\to\infty$ completes the proof.
\end{proof}

We will use the following bounds for the Carleson tent maximal operator. The arguments are standard but included for completeness.
\begin{lemma}\label{MaximalWeakType}
    If $p \in (1,\infty)$ and $\sigma \in B_p$, then 
    $$
        \|M\|_{L^p(\mathbb{H},\sigma)\rightarrow L^p(\mathbb{H},\sigma)} \lesssim [\sigma]_{B_p}^{\frac{p'}{p}}.
    $$
    If $p \in [1,\infty)$ and $\sigma \in B_p$, then 
    $$
        \|M\|_{L^p(\mathbb{H},\sigma)\rightarrow L^{p,\infty}(\mathbb{H},\sigma)} \lesssim [\sigma]_{B_p}^{\frac{1}{p}}.
    $$
\end{lemma}

\begin{proof}
Let $f \in L^p(\mathbb{H}, \sigma)$ and assume without loss of generality $f \geq 0.$ It is enough to prove the bound for the Carleson tent maximal operator restricted to a fixed dyadic grid $\mathcal D$, $M^{\mathcal{D}}$. More generally, given a weight $\sigma$, we define the corresponding weighted maximal operator
$$
    M^{\mathcal{D}}_{\sigma}f(z):= \sup_{I \in \mathcal{D}} \langle |f|\rangle_{\sigma, Q_I}\chi_{Q_I}(z), 
$$
where $\langle f\rangle_{\sigma, Q_I} := \frac{1}{\sigma(Q_I)}\int_{Q_I} f \sigma\,dA$. Standard arguments show that $M^{\mathcal{D}}_{\sigma}$ is bounded on $L^q(\mathbb H, \sigma)$ for any $1<q<\infty$ with constant independent of the weight $\sigma$. 

The strong-type estimate then follows from the pointwise bound
$$ 
    M^{\mathcal{D}}f(z) \leq [\sigma]_{B_p}^{\frac{p'}{p}} \left(M^{\mathcal{D}}_{\sigma}(\sigma^{-1}(M^{\mathcal{D}}_\omega (f \omega^{-1}) )^{p-1}  )(z)  \right )^{\frac{p'}{p}}
$$
for $f \in L^1(\mathbb H)$ and almost every $z \in \mathbb H$, where $\omega := \sigma^{-\frac{p'}{p}}$ is the dual weight of $\sigma$. The bounds for the weighted maximal functions $M^{\mathcal{D}}_\sigma$ and $M^{\mathcal{D}}_\omega$ on $L^{p'}(\mathbb H, \sigma)$ and $L^p(\mathbb H, \omega)$, respectively, then yield the desired strong-type bound. 

To prove the weak-type $(p,p)$ estimate, fix $\lambda>0$ and consider 
$$ 
    E_\lambda:= \{z \in \mathbb H \colon M^{\mathcal{D}}f(z)>\lambda\}.
$$ 
Let $\{I_j\}$ be the collection of maximal dyadic intervals in $\mathcal{D}$  such that $\langle f \rangle_{Q_{I_j}}>\lambda$ and notice that $\{ Q_{I_j}\}$ is a disjoint cover for $E_\lambda$. We estimate
\begin{align*}
\sigma(E_\lambda) & \leq \sum_j \sigma(Q_{I_j}) \frac{\langle f \rangle_{Q_{I_j}}^p}{\lambda^p} \\
& \leq \frac{1}{\lambda^p} \sum_{j} \langle \sigma \rangle_{Q_{I_j}} \langle \sigma^{-\frac{p'}{p}} \rangle_{Q_{I_j}}^{\frac{p}{p'}} \int_{Q_{I_j}} |f|^p \sigma \, dA \\
& \leq \frac{[\sigma]_{B_p}}{\lambda^p} \sum_j \int_{Q_{I_j}} |f|^p \sigma \, dA\\
& \leq \frac{[\sigma]_{B_p}}{\lambda^p} \|f\|_{L^p(\mathbb H, \sigma)}^p,
\end{align*}
which proves the desired bound. 
\end{proof}

The following lemma justifies the appearance of $[uv]_{B_{\infty}}$ in Theorem \ref{MixedWeak}.
\begin{lemma}\label{uvBpLemma}
If $p \in [1,\infty)$, $u\in B_{1}$, and $v\in B_{p}(u)$, then $uv\in B_{p}$ with 
$$
    [uv]_{B_{p}}\leq[u]_{B_{1}}^p[v]_{B_{p}(u)}.
$$
\end{lemma}
\noindent As the proof of Lemma \ref{uvBpLemma} is a simple adaptation of that of \cite{CuMP2005}*{Lemma 2.1}, we omit it.\looseness=-1


\section{Main Results}\label{GeneralSection}

\subsection{First proof of Theorem \ref{WeightedWeakL1}}

We give our first proof of Theorem \ref{WeightedWeakL1}, which follows the ideas in the proof of \cite{NS2024}*{Theorem A}. 
\begin{proof}[Proof of Theorem \ref{WeightedWeakL1}]
By Proposition \ref{SparseBound}, it suffices to prove the bound with $P$ replaced by $S_{\mathcal{D}}$ for an arbitrary dyadic grid $\mathcal{D}$. Let $f \in L^1(\mathbb{H},\sigma)$ and assume without loss of generality that $f$ is nonnegative, bounded, and normalized in $L^1(\mathbb{H},\sigma)$. Appealing to \cite{Grafakos1}*{Exercise 1.4.14}, we have the equivalence
$$
    \|\mathcal{S}_{\mathcal{D}}f\|_{L^{1,\infty}(\mathbb{H},\sigma)} \approx \sup_{\substack{E \subseteq \mathbb{H}\\ 0< \sigma(E)<\infty}}\inf_{\substack{E' \subseteq E \\ \sigma(E') \ge \frac{1}{2}\sigma(E)}} \langle \mathcal{S}_{\mathcal{D}}f, \sigma\chi_{E'}\rangle.
$$
Fix $E \subseteq \mathbb{H}$ with $0< \sigma(E)<\infty$ and put 
$$
    E':= \{z \in E: M^{\mathcal{D}}f(z) \leq \gamma\},
$$
where $\gamma:= \frac{2 [\sigma]_{B_1}}{\sigma(E)}$. By Lemma \ref{MaximalWeakType}, we have that $\sigma(E') \ge \frac{1}{2}\sigma(E)$, and so it suffices to prove
\begin{align}\label{MainEstimate}
    \langle\mathcal{S}_{\mathcal{D}}f,\sigma\chi_{E'}\rangle = \sum_{I \in \mathcal{D}}\langle f\rangle_{Q_I}\langle \sigma\chi_{E'}\rangle_{Q_I}A(Q_I)\lesssim [\sigma]_{B_1}(\log [\sigma]_{B_{\infty}}+1).
\end{align}

If $I \in \mathcal{D}$ satisfies $\langle f\rangle_{Q_I}> \gamma$, then $Q_I \cap E' = \emptyset$, so we only need to bound the sum over intervals in $\mathcal{D}_+:= \{I \in \mathcal{D} \colon \langle f\rangle_{Q_I}\leq \gamma\}$. Given $\lambda \in (0,1)$, set $\mathcal{D}_{\lambda}:= \{I \in \mathcal{D}_+\colon \sigma(E'\cap Q_I)> \lambda\sigma(Q_I)\}$ and $\mathcal{D}_{\lambda}^*:=\{I \in \mathcal{D}_{\lambda}\colon I \text{ is maximal with respect to inclusion}\}$. Then 
\begin{align*}
    \sum_{I \in \mathcal{D}_+}\langle f\rangle_{Q_I}&\langle \sigma\chi_{E'}\rangle_{Q_I}A(Q_I) = \sum_{I \in \mathcal{D}_+}\langle f\rangle_{Q_I}\sigma(E' \cap Q_I)\\
    &= \sum_{I \in \mathcal{D}_+}\langle f\rangle_{Q_I}\sigma(Q_I)\int_0^{\frac{\sigma(E'\cap Q_I)}{\sigma(Q_I)}} \,d\lambda\\
    &= \int_0^1\sum_{I \in \mathcal{D}_{\lambda}}\langle f\rangle_{Q_I}\sigma(Q_I)\,d\lambda\\
    &= \int_0^1\sum_{I_0 \in \mathcal{D}_{\lambda}^*}\sum_{\substack{I \in \mathcal{D}_{\lambda}\\ I \subseteq I_0}} \langle f\rangle_{Q_I}\sigma(Q_I)\,d\lambda.
\end{align*}

We estimate the inner summation above. Fix $I_0 \in \mathcal{D}_{\lambda}^*$ and $\lambda \in (0,1)$. Using H\"older's inequality, we have for any $q>1$ that
$$
    \sum_{\substack{I \in \mathcal{D}_{\lambda}\\ I \subseteq I_0}} \langle f\rangle_{Q_I}\sigma(Q_I) \leq \bigg(\sum_{\substack{I \in \mathcal{D}_{\lambda} \\ I \subseteq I_0}}\langle f\rangle_{Q_I}^q\sigma(Q_I)\bigg)^{\frac{1}{q}}\bigg(\sum_{\substack{I \in \mathcal{D}_{\lambda} \\ I \subseteq I_0}}\sigma(Q_I)\bigg)^{\frac{1}{q'}}.
$$
Using Lemma \ref{cup}, we have 
$$
    \sum_{\substack{I \in \mathcal{D}_{\lambda} \\ I \subseteq I_0}}\sigma(Q_I) \leq 2[\sigma]_{B_{\infty}}\sigma(Q_{I_0}),
$$
and, therefore 
$$
    \sum_{\substack{I \in \mathcal{D}_{\lambda}\\ I \subseteq I_0}} \langle f\rangle_{Q_I}\sigma(Q_I) \leq 2^{\frac{1}{q'}}[\sigma]_{B_{\infty}}^{\frac{1}{q'}}\sigma(Q_{I_0})^{\frac{1}{q'}}\bigg(\sum_{\substack{I \in \mathcal{D}_{\lambda} \\ I \subseteq I_0}}\langle f\rangle_{Q_I}^q\sigma(Q_I)\bigg)^{\frac{1}{q}}.
$$

Now, decompose $\mathcal{D}_{\lambda} = \bigcup_{k=0}^{\infty}\mathcal{E}_k$, where
$$
    \mathcal{E}_k:= \{I \in \mathcal{D}_{\lambda}: C^{-(k+1)}\gamma < \langle f\rangle_{Q_I}\leq C^{-k}\gamma\}
$$
for a fixed $C \in (1,2)$. For $I \in \mathcal{E}_k$, let 
$$
    E_I := Q_I\setminus \bigcup_{\substack{I' \in \mathcal{E}_k\\ I' \subsetneq I}} Q_{I'}.
$$
Letting $\text{ch}_{\mathcal{E}_k}(I)$ denote the collection of maximal $I'\in\mathcal{E}_k$ that are properly contained in $I$, we have that, by \eqref{SparseCondition}, if $I \in \mathcal{E}_k$, then
\begin{align*}
    \int_{Q_I} f\,dA &= \int_{E_I} f\,dA + \sum_{I' \in \text{ch}_{\mathcal{E}_k}(I)} \int_{Q_{I'}} f\,dA\\
    &\leq \int_{E_I} f\,dA + C^{-k}\gamma\sum_{I' \in \text{ch}_{\mathcal{E}_k}(I)} A(Q_{I'})\\
    &\leq \int_{E_I} f\,dA + C^{-k}\gamma\frac{1}{2} A(Q_I)\\
    &\leq \int_{E_I} f\,dA + \frac{C}{2} \int_{Q_I} f\,dA,
\end{align*}
which, since $C <2$, gives 
\begin{align}\label{MagicLemma}
    \int_{Q_I} f\,dA \lesssim \int_{E_I} f\,dA.
\end{align}
Using the above estimate, the definition of $\mathcal{E}_k$, \eqref{MagicLemma}, and the disjointness of the $E_I$, we have
\begin{align*}
    \sum_{\substack{I \in \mathcal{D}_{\lambda} \\ I \subseteq I_0}}\langle f\rangle_{Q_I}\sigma(Q_I) &\leq 2^{\frac{1}{q'}}[\sigma]_{B_{\infty}}^{\frac{1}{q'}}\sigma(Q_{I_0})^{\frac{1}{q'}}\bigg(\sum_{k=0}^{\infty}\sum_{\substack{I \in \mathcal{E}_k\\ I \subseteq I_0}}\langle f\rangle_{Q_I}^q\sigma(Q_I)\bigg)^{\frac{1}{q}}\\
    &\lesssim [\sigma]_{B_{\infty}}^{\frac{1}{q'}}\sigma(Q_{I_0})^{\frac{1}{q'}}\bigg(\sum_{k=0}^{\infty} C^{-k(q-1)}\gamma^{q-1}\sum_{\substack{I \in \mathcal{E}_k \\ I \subseteq I_0}} \langle f\rangle_{Q_I}\sigma(Q_I)\bigg)^{\frac{1}{q}}\\
    &\lesssim [\sigma]_{B_{\infty}}^{\frac{1}{q'}}\gamma^{\frac{1}{q'}}\sigma(Q_{I_0})^{\frac{1}{q'}}\bigg(\sum_{k=0}^{\infty}C^{-k(q-1)}\sum_{\substack{I \in \mathcal{E}_k\\ I \subseteq I_0}}\int_{E_I}f M(\sigma\chi_{Q_{I_0}})\,dA\bigg)^{\frac{1}{q}}\\
    &\lesssim [\sigma]_{B_{\infty}}^{\frac{1}{q'}}[\sigma]_{B_1}^{\frac{1}{q'}}\sigma(E)^{-\frac{1}{q'}}\sigma(Q_{I_0})^{\frac{1}{q'}}(1-C^{1-q})^{-\frac{1}{q}}\|f\|_{L^1(Q_{I_0},M(\sigma\chi_{Q_{I_0}}))}^{\frac{1}{q}}.
\end{align*}
Combining this with the first bound and using $(1-C^{1-q})^{-\frac{1}{q}}\approx q'$ and $\sigma(Q_{I_0})\leq \frac{1}{\lambda}\sigma(E'\cap Q_{I_0})$, we have an estimate of the left-hand side of \eqref{MainEstimate} by a constant times
\begin{align*}
    q'[\sigma]_{B_{\infty}}^{\frac{1}{q'}}[\sigma]_{B_1}^{\frac{1}{q'}}&\sigma(E)^{-\frac{1}{q'}}\int_0^1\sum_{I_0 \in \mathcal{D}_{\lambda}^*}\|f\|_{L^1(Q_{I_0},M\sigma)}^{\frac{1}{q}}\sigma(Q_{I_0})^{\frac{1}{q'}}\,d\lambda\\
    &\leq q'[\sigma]_{B_{\infty}}^{\frac{1}{q'}}[\sigma]_{B_1}^{\frac{1}{q'}}\sigma(E)^{-\frac{1}{q'}}\int_0^1\bigg(\sum_{I_0 \in \mathcal{D}_{\lambda}^*}\|f\|_{L^1(Q_{I_0},M\sigma)}\bigg)^{\frac{1}{q}}\bigg(\sum_{I_0 \in \mathcal{D}_{\lambda}^*}\sigma(Q_{I_0})\bigg)^{\frac{1}{q'}}\,d\lambda\\
    &\leq q'[\sigma]_{B_{\infty}}^{\frac{1}{q'}}[\sigma]_{B_1}^{\frac{1}{q'}}\sigma(E)^{-\frac{1}{q'}}\|f\|_{L^1(\mathbb{H},M\sigma)}^{\frac{1}{q}}\int_0^1\lambda^{-\frac{1}{q'}}\bigg(\sum_{I_0 \in \mathcal{D}_{\lambda}^*}\sigma(E'\cap Q_{I_0})\bigg)^{\frac{1}{q'}}\,d\lambda\\
    &\leq qq'[\sigma]_{B_{\infty}}^{\frac{1}{q'}}[\sigma]_{B_1}^{\frac{1}{q'}}\|f\|_{L^1(\mathbb{H},M\sigma)}^{\frac{1}{q}}\\
    &\leq qq'[\sigma]_{B_{\infty}}^{\frac{1}{q'}}[\sigma]_{B_1}.
\end{align*}
Taking $q'=\log [\sigma]_{B_{\infty}} + 2$, we have $q \leq 2$ and $[\sigma]_{B_{\infty}}^{\frac{1}{q'}}\leq e$, and hence the result follows.
\end{proof}

\begin{rem}\label{p>1Remark}
A slight modification of the above proof of Theorem \ref{WeightedWeakL1} gives the bound 
$$
    \|P^+\|_{L^p(\mathbb{H},\sigma)\rightarrow L^{p,\infty}(\mathbb{H},\sigma)} \lesssim [\sigma]_{B_p}^{\frac{1}{p}}[\sigma]_{B_{\infty}}^{\frac{1}{p'}}(\log [\sigma]_{B_{\infty}}+1)
$$
for $p \in (1,\infty)$ and $\sigma \in B_p$. In particular, the only modifications of the above proof of Theorem \ref{WeightedWeakL1} are using the more general equivalence
$$
    \|\mathcal{S}_{\mathcal{D}}f\|_{L^{p,\infty}(\mathbb{H},\sigma)} \approx \sup_{\substack{E \subseteq \mathbb{H}\\ 0< \sigma(E)<\infty}}\inf_{\substack{E' \subseteq E \\ \sigma(E') \ge \frac{1}{2}\sigma(E)}} \sigma(E)^{-\frac{1}{p'}}\langle \mathcal{S}_{\mathcal{D}}f, \sigma\chi_{E'}\rangle;
$$
defining $\gamma := \left(\frac{2}{\sigma(E)}\right)^{\frac{1}{p}}\|M^{\mathcal{D}}\|_{L^p(\mathbb{H},\sigma) \rightarrow L^{p,\infty}(\mathbb{H},\sigma)}$; using \eqref{MagicLemma}, H\"older's inequality, the $B_p$ condition, and Lemma \ref{cup} to control the inner sum on the second line of the penultimate display by \looseness=-1
\begin{align*}
    \sum_{\substack{I \in \mathcal{E}_k \\ I \subseteq I_0}} \langle f\rangle_{Q_I}\sigma(Q_I) & \lesssim \sum_{\substack{I \in \mathcal{E}_k \\ I \subseteq I_0}} \langle f^p\sigma\rangle_{E_I}^{\frac{1}{p}} \langle \sigma^{-\frac{p'}{p}}\rangle_{Q_I}^{\frac{1}{p'}}\langle\sigma\rangle_{Q_I} A(Q_I)\\
    &\lesssim [\sigma]_{B_p}^{\frac{1}{p}}\sum_{\substack{I \in \mathcal{E}_k \\ I \subseteq I_0}} \langle f^p\sigma\rangle_{E_I}^{\frac{1}{p}}\langle\sigma\rangle_{Q_I}^{\frac{1}{p'}}A(E_I)\\
    &\leq [\sigma]_{B_p}^{\frac{1}{p}}\bigg(\sum_{\substack{I \in \mathcal{E}_k \\ I \subseteq I_0}}\langle f^p\sigma\rangle_{E_I}A(E_I)\bigg)^{\frac{1}{p}}\bigg(\sum_{\substack{I \in \mathcal{E}_k \\ I \subseteq I_0}}\langle \sigma\rangle_{Q_I}A(E_I)\bigg)^{\frac{1}{p'}}\\
    &\lesssim [\sigma]_{B_p}^{\frac{1}{p}}[\sigma]_{B_{\infty}}^{\frac{1}{p'}} \|f\|_{L^p(Q_{I_0},\sigma)}\sigma(Q_{I_0})^{\frac{1}{p'}};
\end{align*}
and bounding the corresponding final display in the proof of Theorem \ref{WeightedWeakL1} using the bound $\|M\|_{L^p(\mathbb{H},\sigma)\rightarrow L^{p,\infty}(\mathbb{H},\sigma)} \lesssim [\sigma]_{B_p}^{\frac{1}{p}}$ of Lemma \ref{MaximalWeakType} and H\"older's inequality with powers $pq$ and $(pq)'$: \looseness=-1
\begin{align*}
    q'[\sigma]_{B_p}^{\frac{1}{pq}}&[\sigma]_{B_{\infty}}^{\frac{1}{q'}+\frac{1}{qp'}}\gamma^{\frac{1}{q'}}\int_0^1\sum_{I_0 \in \mathcal{D}_{\lambda}^*}\|f\|_{L^p(Q_{I_0},\sigma)}^{\frac{1}{q}}\sigma(Q_{I_0})^{\frac{1}{q'}+\frac{1}{qp'}}\,d\lambda\\
    &\lesssim q'[\sigma]_{B_p}^{\frac{1}{pq}+\frac{1}{pq'}}[\sigma]_{B_{\infty}}^{\frac{1}{q'}+\frac{1}{qp'}}\sigma(E)^{-\frac{1}{pq'}}\int_0^1\bigg(\sum_{I_0 \in \mathcal{D}_{\lambda}^*}\|f\|_{L^p(Q_{I_0},\sigma)}^p\bigg)^{\frac{1}{pq}}\bigg(\sum_{I_0 \in \mathcal{D}_{\lambda}^*}\sigma(Q_{I_0})\bigg)^{\frac{pq-1}{pq}}\,d\lambda\\
    &\leq q'[\sigma]_{B_p}^{\frac{1}{p}}[\sigma]_{B_{\infty}}^{\frac{1}{q'}+\frac{1}{qp'}}\sigma(E)^{-\frac{1}{pq'}}\int_0^1\lambda^{-\frac{pq-1}{pq}}\bigg(\sum_{I_0 \in \mathcal{D}_{\lambda}^*}\sigma(E'\cap Q_{I_0})\bigg)^{\frac{pq-1}{pq}}\,d\lambda\\
    &\leq pqq'[\sigma]_{B_p}^{\frac{1}{p}}[\sigma]_{B_{\infty}}^{\frac{1}{q'}+\frac{1}{qp'}}\sigma(E)^{\frac{1}{p'}}.
\end{align*}
The proof is finished by choosing $q$ such that $q'=\log [\sigma]_{B_{\infty}} + 2$, as before.
\end{rem}

\subsection{Proof of Theorem \ref{WeightedWeakLp} and Corollary \ref{StrongPCorollary}}

Our proof strategy essentially follows that of Hyt\"onen and Li in \cite{HL2018}. We will need several lemmata, the first of which is known as Kolmogorov's lemma; see \cite{Grafakos2024}*{Theorem 1.2.10}.
\begin{lemma}\label{lem:kolmogorov}
Let $(S,\mu)$ be a $\sigma$-finite measure space, $f$ be a measurable function on $S$, and $p \in(0,\infty)$. Then $f\in L^{p,\infty}(S,\mu)$ if and only if there exists $C>0$ and $\theta \in (0,p)$ such that 
\[
\int_E\!|f|^\theta\,\mathrm{d}\mu\leq\tfrac{p}{p-\theta}C^\theta\mu(E)^{1-\frac{\theta}{p}}
\]
for all $E\subseteq S$ with $0<\mu(E)<\infty$, and the optimal constant $C>0$ satisfies
\[
C\leq \|f\|_{L^{p,\infty}(S,\mu)}\leq \big(\tfrac{p}{p-\theta}\big)^{\frac{1}{\theta}}C.
\]
\end{lemma}

The following lemma is a variant of the testing condition of \cite{LSU2009}. Our proof follows an alternative strategy presented in \cite{NSS2024}.
\begin{lemma}\label{lem:lsutesting}
If $\mathcal{D}$ is a dyadic grid, $\sigma$ is a weight, and $p \in (1,\infty)$, then
\begin{equation}\label{eq:thmdtesting}
\|S_{\mathcal{D}}\|_{L^p(\mathbb{H},\sigma)\to L^{p,\infty}(\mathbb{H},\sigma)} \lesssim \sup_{I_0\in\mc{D}}\sigma(Q_{I_0})^{-\frac{1}{p'}}\bigg\|\sum_{\substack{I \in\mc{D}\\ I \subseteq I_0}}\langle\sigma\rangle_{Q_I}\ind_{Q_I}\bigg\|_{L^{p'}(\mathbb{H}, \sigma^{1-p'})}.
\end{equation}
\end{lemma}
\begin{proof}
Let $f \in L^p(\mathbb{H},\sigma)$ and assume without loss of generality that $f$ is nonnegative and bounded. By Lemma~\ref{lem:kolmogorov} with $\theta=1$, it suffices to prove that for any $E\subseteq\mathbb{H}$ with $0<\sigma(E)<\infty$, we have 
\[
    \int_{E}\!|S_{\mc{D}}f|\sigma\,\mathrm{d}A \leq \sum_{I \in\mc{D}}\langle f\rangle_{Q_I}\sigma(E\cap Q_I )\lesssim \mc{M}\|f\|_{L^p(\mathbb{H},\sigma)}\sigma(E)^\frac{1}{p'}.
\]
Fix $\lambda>0$, let
\[
\mc{D}_\lambda:=\big\{I\in\mc{D}\colon \sigma(E\cap Q_I )>\lambda\sigma(Q_I)\big\},
\]
and let $\mc{D}^\ast_\lambda$ denote the $I \in\mc{D}_\lambda$ for which $Q_I$ is maximal with respect to inclusion. Note that 
\[
\sum_{I_0\in\mc{D}^\ast_\lambda}\ind_{Q_{I_0}}\leq\ind_{\{M^{\mc{D}}_\sigma(\ind_E)>\lambda\}}
\]
almost everywhere, and thus, setting 
\[
\mc{M}:=\sup_{I_0\in\mc{D}}\sigma(Q_{I_0})^{-\frac{1}{p'}}\bigg\|\sum_{\substack{I \in\mc{D}\\ I \subseteq I_0}}\langle \sigma\rangle_{Q_I}\ind_{Q_I}\bigg\|_{L^{p'}(\mathbb{H},\sigma^{1-p'})},
\]
we have
\begin{align*}
\sum_{I \in\mc{D}_\lambda}\langle f\rangle_{Q_I}\sigma(Q_I)
&=\sum_{I_0\in\mc{D}^\ast_\lambda}\sum_{\substack{I \in\mc{D}_\lambda\\ I \subseteq I_0}}\langle f\rangle_{Q_I}\sigma(Q_I)\\
&=\sum_{I_0\in\mc{D}^\ast_\lambda}\int_{Q_{I_0}}\!f\sum_{\substack{I \in\mc{D}_\lambda\\ I \subseteq I_0}}\langle\sigma\rangle_{Q_I}\ind_{Q_I}\,\mathrm{d}A\\
&\leq\mc{M}\sum_{I_0\in\mc{D}^\ast_\lambda}\bigg(\int_{Q_{I_0}}\!f^p \sigma\,\mathrm{d}A\bigg)^{\frac{1}{p}}\sigma(Q_{I_0})^{\frac{1}{p'}}\\
&\leq\mc{M}\bigg(\sum_{I_0\in\mc{D}^\ast_\lambda} \int_{Q_{I_0}}\!f^p \sigma\,\mathrm{d}A\bigg)^{\frac{1}{p}}\bigg(\sum_{I_0\in\mc{D}^\ast_\lambda}\sigma(Q_{I_0})\bigg)^{\frac{1}{p'}}\\
&\leq\mc{M}\|f\|_{L^p(\mathbb{H},\sigma)}\sigma(\big(\{M^{\mc{D}}_\sigma(\ind_E)>\lambda\}\big)^{\frac{1}{p'}},
\end{align*}
where we have used H\"older's inequality in both the first and the second inequalities. Using the above estimate and the boundedness of $M^{\mathcal{D}}_\sigma$ on $L^{p',1}(\mathbb{H},\sigma)$ (which follows from the Marcinkiewicz interpolation theorem), we have 
\begin{align*}
\sum_{I \in\mc{D}}\langle f\rangle_{Q_I}\sigma(E\cap Q_I)
&=\sum_{I \in\mc{D}}\langle f\rangle_{Q_I}\bigg(\int_0^{\frac{\sigma(E\cap Q_I)}{\sigma(Q_I)}}\!\,\mathrm{d}\lambda\bigg)\sigma(Q_I)\\
&=\int_0^{\infty} \sum_{I \in\mc{D}_\lambda}\langle f\rangle_{Q_I}\sigma(Q_I)\,\mathrm{d}\lambda\\
&\leq\mc{M}\|f\|_{L^p(\mathbb{H},\sigma)}\|M^{\mc{D}}_\sigma(\ind_E)\|_{L^{p',1}(\mathbb{H},\sigma)}\\
&\lesssim\mc{M}\|f\|_{L^p(\mathbb{H},\sigma)}\sigma(E)^{\frac{1}{p'}}.
\end{align*}
This proves the assertion.
\end{proof}

\begin{lemma}\label{lem:multilineartesting3}
If $\mathcal{D}$ is a dyadic grid and $\alpha_1,\alpha_2\in[0,1)$ satisfy $\alpha_1+\alpha_2<1$, then 
\[
\sum_{\substack{I\in\mc{D}\\ I\subseteq J}}\langle \sigma\rangle_{Q_I}^{\alpha_1}\langle\omega\rangle_{Q_I}^{\alpha_2}A(Q_I)\lesssim\langle \sigma\rangle_{Q_J}^{\alpha_1}\langle\omega\rangle_{Q_J}^{\alpha_2}A(Q_J)
\]
for all weights $\sigma$ and $\omega$ and all $J\in\mc{D}$.
\end{lemma}
\begin{proof}
Set $\theta_j:=\alpha_j+\tfrac{1}{2}(1-\alpha_1-\alpha_2)>\alpha_j$ for $j=1,2$, and note that $\theta_1+\theta_2=1$. 
Using \eqref{SparseCondition}, H\"older's inequality, Lemma~\ref{lem:kolmogorov} with $p=1$, $\theta=\tfrac{\alpha_1}{\theta_1},\tfrac{\alpha_2}{\theta_2}$, and $\|M^{\mc{D}(I)}\|_{L^1(Q_I)\to L^{1,\infty}(Q_I)}\leq 1$, we have
\begin{align*}
\sum_{\substack{I\in\mc{D}\\ I\subseteq J}}&\langle \sigma\rangle_{Q_I}^{\alpha_1}\langle\omega\rangle_{Q_I}^{\alpha_2}A(Q_I) \lesssim \int_{Q_J}\!(M^{\mc{D}}(\sigma \chi_{Q_J}))^{\alpha_1}(M^{\mc{D}}(\omega \chi_{Q_J}))^{\alpha_2}\,\mathrm{d}A\\
&
\leq\bigg(\int_{Q_J}\!(M^{\mc{D}}(\sigma \chi_{Q_J}))^{\frac{\alpha_1}{\theta_1}}\,\mathrm{d}A\bigg)^{\theta_1}\bigg(\int_{Q_J}\!(M^{\mc{D}}(\omega \chi_{Q_J}))^{\frac{\alpha_2}{\theta_2}}\,\mathrm{d}A\bigg)^{\theta_2}\\
&\lesssim\Big(\|M^{\mc{D}}(\sigma \chi_{Q_J})\|_{L^{1,\infty}(Q_J)}^{\frac{\alpha_1}{\theta_1}}A(Q_J)^{1-\frac{\alpha_1}{\theta_1}}\Big)^{\theta_1}
\Big(\|M^{\mc{D}}(\omega \chi_{Q_J}) \|_{L^{1,\infty}(Q_J)}^{\frac{\alpha_2}{\theta_2}}A(Q_J)^{1-\frac{\alpha_2}{\theta_2}}\Big)^{\theta_2}\\
&\lesssim \Big(\|\sigma\|_{L^1(Q_J)}^{\alpha_1}A(Q_J)^{-\alpha_1}\Big)\Big(\|\omega\|_{L^1(Q_J)}^{\alpha_2}A(Q_J)^{-\alpha_2}\big)A(Q_J)^{\theta_1+\theta_2}\\
&=\langle \sigma\rangle_{Q_J}^{\alpha_1}\langle\omega\rangle_{Q_J}^{\alpha_2}A(Q_J).
\end{align*}
This proves the result.
\end{proof}

The next lemma uses \cite{COV2004}*{Proposition~2.2}, which shows that
\begin{equation}\label{eq:covlemma}
\bigg\|\sum_{I\in\mc{F}}a_{I}\ind_{Q_I}\bigg\|_{L^q(\mathbb{H},v)} \approx \bigg(\sum_{J\in\mc{F}}\bigg(\frac{1}{v(Q_J)}\sum_{\substack{I\in\mc{F}\\ I \subseteq J}}a_{I}v(Q_I)\bigg)^{q-1}a_{J}v(Q_J)\bigg)^{\frac{1}{q}}
\end{equation}
for any $q\in[1,\infty)$, weight $v$, collection $\mc{F}\subseteq\mc{D}$, and $\{a_{I}\}_{I\in\mc{F}}\subseteq [0,\infty)$. For this, note that the collection $(Q_I)_{I\in\mc{F}}$ is a subcollection of a dyadic grid in the plane.

\begin{lemma}\label{lem:thme}
If $\mathcal{D}$ is a dyadic grid, $p \in (1,\infty)$, and $\sigma\in B_p$, then
\[
\bigg\|\sum_{I\in\mc{E}}\langle \sigma\rangle_{Q_I}\ind_{Q_I}\bigg\|_{L^{p'}(\mathbb{H}, \sigma^{1-p'})}\lesssim [\sigma]_{B_p}^{\frac{1}{p}}\bigg(\sum_{I\in\mc{E}}\sigma(Q_I)\bigg)^{\frac{1}{p'}}
\]
for any $\mc{E}\subseteq\mc{D}$.
\end{lemma}
\begin{proof}
Set $\omega:=\sigma^{1-p'}$ and let $\gamma:=\min\{p,p'\}$ so that $1<\gamma\leq2$. Since $1-\frac{\gamma}{p}+1-\frac{\gamma}{p'}=2-\gamma\in[0,1)$, it follows from Lemma~\ref{lem:multilineartesting3} that
\begin{align*}
\frac{1}{\omega(Q_J)}\sum_{\substack{I\in\mc{E}\\ I \subseteq J}}\langle\sigma\rangle_{Q_I}\omega(Q_I)
&\leq[\sigma]^{\frac{\gamma}{p}}_{B_p}\frac{1}{\omega(Q_J)}\sum_{\substack{I\in\mc{E}\\ I\subseteq J}}\langle \sigma\rangle^{1-\frac{\gamma}{p}}_{Q_I}\langle \omega\rangle^{1-\frac{\gamma}{p'}}_{Q_I}A(Q_I)\\
&\lesssim [\sigma]^{\frac{\gamma}{p}}_{B_p}\frac{1}{\omega(Q_J)}\langle \sigma\rangle^{1-\frac{\gamma}{p}}_{Q_J}\langle \omega\rangle^{1-\frac{\gamma}{p'}}_{Q_J}A(Q_J)\\
&=[\sigma]^{\frac{\gamma}{p}}_{B_p}\langle \sigma\rangle^{1-\frac{\gamma}{p}}_{Q_J}\langle \omega\rangle^{-\frac{\gamma}{p'}}_{Q_J}.
\end{align*}

Thus, by \eqref{eq:covlemma},
\begin{align*}
\bigg\|\sum_{I \in\mc{E}}\langle \sigma\rangle_{Q_I}\ind_{Q_I}\bigg\|_{L^{p'}(\mathbb{H},\omega)}
&\lesssim[\sigma]^{\frac{\gamma}{p^2}}_{B_p}\bigg(\sum_{I \in\mc{E}}\langle \sigma\rangle^{(p'-1)(1-\frac{\gamma}{p})}_{Q_I}\langle \omega\rangle^{1-\frac{\gamma}{p}}_{Q_I}\sigma(Q_I)\bigg)^{\frac{1}{p'}}\\
&\leq [\sigma]^{\frac{\gamma}{p^2}+\frac{p'-1}{p'}(1-\frac{\gamma}{p})}_{B_p}\bigg(\sum_{I \in\mc{E}}\sigma(Q_I)\bigg)^{\frac{1}{p'}}.
\end{align*}
Since $\frac{\gamma}{p^2}+\frac{p'-1}{p'}(1-\frac{\gamma}{p})=\frac{1}{p}$, the assertion follows.
\end{proof}

\begin{proof}[Proof of Theorem \ref{WeightedWeakLp}]
We first prove the upper bound. By Proposition \ref{SparseBound}, it suffices to prove the bound with $P$ replaced by $S_{\mathcal{D}}$ for an arbitrary dyadic grid $\mathcal{D}$. Then, by Lemma~\ref{lem:lsutesting}, it suffices to show that
\[
\bigg\|\sum_{\substack{I \in\mc{D}\\ I \subseteq I_0}}\langle\sigma\rangle_{Q_I}\ind_{Q_I}\bigg\|_{L^{p'}(\mathbb{H}, \sigma^{1-p'})}\lesssim [\sigma]_{B_p}^{\frac{1}{p}}[\sigma]_{B_{\infty}}^{\frac{1}{p'}}\sigma(Q_{I_0})^{\frac{1}{p'}}
\]
uniformly for all $I_0\in\mc{D}$. Indeed, by Lemma~\ref{lem:thme} and Lemma~\ref{cup}, we have
\[
\bigg\|\sum_{\substack{I \in\mc{D}\\ I \subseteq I_0}}\langle \sigma\rangle_{Q_I}\ind_{Q_I}\bigg\|_{L^{p'}(\mathbb{H},\sigma^{1-p'})}\lesssim[\sigma]_{B_p}^{\frac{1}{p}}\bigg(\sum_{\substack{I\in\mc{D}\\ I\subseteq I_0}}\sigma(Q_I)\bigg)^{\frac{1}{p'}}\lesssim[\sigma]_{B_p}^{\frac{1}{p}}[\sigma]_{B_\infty}^{\frac{1}{p'}}\sigma(Q_{I_0})^{\frac{1}{p'}}.
\]
The result follows.

\bigskip

We now show that the estimate is sharp in terms of $[\sigma]_{B_p}$. Our argument is inspired by the examples from \cite{PR2013}*{Section 5} and \cite{R2025}*{Chapter 1}. Suppose $\phi\colon [1,\infty)\to [1,\infty)$ is an increasing function such that 
\[
\|P\|_{L^p(\mathbb{H},\sigma)\to L^{p,\infty}(\mathbb{H},\sigma)}\leq\phi([\sigma]_{B_p})
\]
for all $\sigma\in B_p$. Let $p \in (1,\infty)$, $\delta \in (0,1]$, $\sigma(z) := |z|^{2(\delta-1)}$, 
\[
S:= \bigg\{z \in \mathbb{C}\colon |z|>2 \,\,\text{and}\,\, \frac{3\pi}{8} < \text{Arg}(z)< \frac{5\pi}{8}\bigg\},
\]
and $f(z) := |z|^{2\delta(1-p')}\chi_{S}(z)$. Then $\sigma \in B_p$ and $f \in L^{p}(\mathbb{H},\sigma)$ with 
$$
    [\sigma]_{B_p}\approx \delta^{-1} \quad\text{and}\quad \|f\|_{L^p(\mathbb{H},\sigma)} \approx \delta^{-\frac{1}{p}}.
$$

Note that if $z \in \mathbb{H} \cap D(0,1)$ and $w \in S$, then, since $|w|>2\geq2|z|$, we have
\[
    |z-\overline{w}|\leq|z|+|w|\leq\tfrac{3}{2}|w|.
\]
Moreover, a geometric argument shows that for $z\in\mathbb{H}\cap D(0,1)$ and $w\in S$,
$$
\text{Arg}(\overline{z}-w)\geq \text{Arg}(-1-2\text{e}^{\text{i}\frac{3\pi}{8}})
=\pi+\arctan\frac{2\sin\frac{3\pi}{8}}{1+2\cos\frac{3\pi}{8}}
=\pi+\arctan\frac{\sqrt{2+\sqrt{2}}}{1+\sqrt{2-\sqrt{2}}},
$$
and
\begin{align*}
\text{Arg}(\overline{z}-w)
&\leq \text{Arg}(1-2\text{e}^{\text{i}\frac{5\pi}{8}})
=\text{Arg}(1+2\text{e}^{-\text{i}\frac{3\pi}{8}})\\
&=2\pi-\arctan\frac{2\sin\frac{3\pi}{8}}{1+2\cos\frac{3\pi}{8}}
=2\pi-\arctan\frac{\sqrt{2+\sqrt{2}}}{1+\sqrt{2-\sqrt{2}}}.
\end{align*}
Since $\arctan\frac{\sqrt{2+\sqrt{2}}}{1+\sqrt{2-\sqrt{2}}}\approx0.8082>\frac{49\pi}{192}$, we have
$$
\pi+\frac{49\pi}{192}\leq\text{Arg}(\overline{z}-w)\leq2\pi-\frac{49\pi}{192},
$$
which implies
$$
\frac{\pi}{2}<\frac{49\pi}{96}\leq\text{Arg}\big((\overline{z}-w)^2\big)
\leq\frac{143\pi}{96}<\frac{3\pi}{2}.
$$
Consequently,
$$
\text{Re}\big((\overline{z}-w)^2\big)<0,
$$
and
$|\text{Im}\big((\overline{z}-w)^2\big)|\leq \left|\tan\frac{49\pi}{96}\right||\text{Re}\big((\overline{z}-w)^2\big)|<31|\text{Re}\big((\overline{z}-w)^2\big)|$ so that
$$
|\overline{z}-w|^2\leq 32|\text{Re}\big((\overline{z}-w)^2\big)|=-32\text{Re}\big((\overline{z}-w)^2\big).
$$
Hence, combining the above estimates, we see that
$$
-\text{Re}\bigg(\frac{1}{(z-\overline{w})^2}\bigg) = \frac{-\text{Re}\big((\overline{z}-w)^2\big)}{|z-\overline{w}|^4} \geq\frac{1}{32} \frac{|\overline{z}-w|^2}{|z-\overline{w}|^4} = \frac{1}{32}\frac{1}{|z-\overline{w}|^2}\geq\frac{1}{72}\frac{1}{|w|^2}.
$$
Therefore, for $z \in \mathbb{H}\cap D(0,1)$, we have 
\begin{align*}
    |Pf(z)| &= \bigg|\int_S \frac{|w|^{2\delta(1-p')}}{(z-\overline{w})^2} \, dA(w)\bigg|\\
    &\geq \bigg|\text{Re}\bigg(\int_S \frac{|w|^{2\delta(1-p')}}{(z-\overline{w})^2}\,dA(w)\bigg)\bigg|\\
    &= -\int_S \text{Re}\bigg(\frac{1}{(z-\overline{w})^2}\bigg)|w|^{2\delta(1-p')}\,dA(w)\\
    &\gtrsim\int_S |w|^{2\delta(1-p')-2}\,dA(w)\\
    &\approx \int_2^{\infty} r^{2\delta(1-p')}\,\frac{dr}{r}
    \approx \delta^{-1}.
\end{align*}

Denoting the implicit constant in the above estimate by $C>0$, letting $t\geq 1$, and setting $\delta:=t^{-1}$ (so that $[\sigma]_{B_p}\leq ct$ for some $c>0$), we conclude that
\begin{align*}
\phi(ct)&\geq\phi([\sigma]_{B_p})\geq\frac{\|Pf\|_{L^{p,\infty}(\mathbb{H},\sigma)}}{\|f\|_{L^p(\mathbb{H},\sigma)}}\\
&\gtrsim\delta^{\frac{1}{p}} \sup_{0<\lambda < C\delta^{-1}} \lambda \sigma(\{z \in \mathbb{H}\cap D(0,1)\colon |Pf(z)|>\lambda\})^{\frac{1}{p}}\\
    &= C\delta^{-1} \delta^{\frac{1}{p}}\sigma(\mathbb{H}\cap D(0,1))^{\frac{1}{p}}
    \approx \delta^{-1} = t.
\end{align*}
The result follows.
\end{proof}

\begin{proof}[Proof of Corollary \ref{StrongPCorollary}]
Let $1\leq q< s < p < r < \infty$. Since $\sigma\in B_q\subseteq B_s\subseteq B_r$, we have by Theorem \ref{WeightedWeakLp} that 
$$
    \|P^+\|_{L^s(\mathbb{H},\sigma)\rightarrow L^{s,\infty}(\mathbb{H},\sigma)} \lesssim [\sigma]_{B_s}^{\frac{1}{s}}[\sigma]_{B_{\infty}}^{\frac{1}{s'}} \quad\text{and}\quad \|P^+\|_{L^r(\mathbb{H},\sigma)\rightarrow L^{r,\infty}(\mathbb{H},\sigma)} \lesssim [\sigma]_{B_r}^{\frac{1}{r}}[\sigma]_{B_{\infty}}^{\frac{1}{r'}}.
$$
By Marcinkiewicz interpolation \cite{Grafakos2024}*{Theorem 1.3.3} and using $[\sigma]_{B_r}\leq[\sigma]_{B_s}\leq [\sigma]_{B_q}$, this implies 
$$
    \|P^+\|_{L^p(\mathbb{H},\sigma)\rightarrow L^p(\mathbb{H},\sigma)} \lesssim \Big([\sigma]_{B_s}^{\frac{1}{s}}[\sigma]_{B_{\infty}}^{\frac{1}{s'}}\Big)^{\frac{\frac{1}{p}-\frac{1}{r}}{\frac{1}{s}-\frac{1}{r}}}\Big([\sigma]_{B_r}^{\frac{1}{r}}[\sigma]_{B_{\infty}}^{\frac{1}{r'}}\Big)^{\frac{\frac{1}{s}-\frac{1}{p}}{\frac{1}{s}-\frac{1}{r}}} \leq [\sigma]_{B_q}^{\frac{1}{p}}[\sigma]_{B_{\infty}}^{\frac{1}{p'}},
$$
as desired.
\end{proof}

\subsection{Proof of Theorem \ref{MixedWeak}}

Before proceeding to the proof of Theorem \ref{MixedWeak}, we provide some more lemmata. Recall that for a function $f$, a weight $\sigma$, and an interval $I\subseteq\mathbb{R}$, we write\looseness=-1
$$
\langle f\rangle_{\sigma,Q_I}:=\frac{1}{\sigma(Q_I)}\int_{Q_I}f\sigma dA.
$$

\begin{lemma}\label{weight-change}
If $u,v$ are weights on $\mathbb{H}$ such that $v\in B_{p}(u)$ for some $p\in[1,\infty)$, then 
$$
\langle|f|\rangle_{u,Q_I}
\leq[v]_{B_{p}(u)}^{\frac{1}{p}}\langle|f|^{p}\rangle_{uv,Q_I}^{\frac{1}{p}}
$$
for any interval $I\subseteq\mathbb{R}$ and any $f\in L^{p}(\mathbb{H},uv)$
\end{lemma}
\begin{proof}
In the case $p=1$, we have
\begin{align*}
\langle|f|\rangle_{u,Q_I}
=\langle v\rangle_{u,Q_I}\frac{1}{uv(Q_I)}\int_{Q_I}|f|u\,dA
\leq [v]_{B_{1}(u)}\langle |f|\rangle_{uv,Q_I}.
\end{align*}
In the case $p>1$, we can apply H\"{o}lder's inequality to deduce that
\begin{align*}
\langle|f|\rangle_{u,Q_I}
&\leq\langle|f|^{p}v\rangle_{u,Q_I}^{\frac{1}{p}}
    \langle v^{-p'/p}\rangle_{u,Q_I}^{\frac{1}{p'}}\\
&=\langle|f|^{p}\rangle_{uv,Q_I}^{\frac{1}{p}}
    \langle v\rangle_{u,Q_I}^{\frac{1}{p}}
    \langle v^{-p'/p}\rangle_{u,Q_I}^{\frac{1}{p'}}\\
&\leq[v]_{B_{p}(u)}^{\frac{1}{p}}\langle|f|^{p}\rangle_{uv,Q_I}^{\frac{1}{p}}.
\end{align*}
The proof is complete.
\end{proof}

\begin{lemma}\label{esi}
If $\sigma\in B_{\infty}$ is of bounded hyperbolic oscillation, then 
$$\frac{\sigma(E)}{\sigma(Q_I)}
\lesssim c_{\sigma}
  \left(\frac{A(E)}{A(Q_I)}\right)^{\frac{1}{1+9[\sigma]_{B_{\infty}}}}
$$
for any interval $I\subset\mathbb{R}$ and any measurable subset $E\subseteq Q_I$.
\end{lemma}
\begin{proof}
Write $r=1+\frac{1}{9[\sigma]_{B_{\infty}}}$. By H\"{o}lder's inequality and \eqref{ReverseHolder}, we have
\begin{align*}
\frac{1}{A(Q_I)}\int_E\sigma \,dA
&\leq\left(\frac{1}{A(Q_I)}\int_E\sigma^rdA\,\right)^{\frac{1}{r}}\left(\frac{A(E)}{A(Q_I)}\right)^{\frac{1}{r'}}\\
&\lesssim c_{\sigma}\left(\frac{1}{A(Q_I)}\int_{Q_I}\sigma \,dA\right)
    \left(\frac{A(E)}{A(Q_I)}\right)^{\frac{1}{1+9[\sigma]_{B_{\infty}}}},
\end{align*}
and the desired inequality follows.
\end{proof}

\begin{lemma}\label{E_I}
Let $\sigma\in B_{\infty}$ be of bounded hyperbolic oscillation and $f\in L^1(\mathbb{H},\sigma)$. If $\mathcal{E}\subseteq\DD$ is a finite collection of dyadic intervals satisfying 
\begin{equation}\label{fj}
2^{-j-1}\leq\langle|f|\rangle_{\sigma,Q_I}\leq 2^{-j}
\end{equation}
for some $j \in \mathbb{N}$ and all $I\in\mathcal{E}$, then for each $I\in\mathcal{E}$ there exists a subset $E_I\subseteq Q_I$ such that
$$
\int_{Q_I}|f|\sigma \,dA\leq 6\int_{E_I}|f|\sigma \,dA
$$
and
$$
\sum_{I\in\mathcal{E}}\chi_{E_I}\lesssim(1+\log c_{\sigma})[\sigma]_{B_{\infty}}.
$$
\end{lemma}
\begin{proof}
Let $\mathcal{E}^0$ be the collection of maximal intervals of $\mathcal{E}$ with respect to inclusion, and for each $i\in\mathbb{N}$, let $\mathcal{E}^i$ be the collection of maximal intervals of $\mathcal{E}\setminus\cup_{k=0}^{i-1}\mathcal{E}^{k}$. 
Fix $I\in\mathcal{E}$. Then there exists a nonnegative integer $i$ such that $I\in\mathcal{E}^i$. For every positive integer $m$, define $L_{I,m}:=\bigcup_{\substack{J\in\mathcal{E}^{i+m}\\J\subseteq I}}Q_J$. It is clear that
$$
A(L_{I,1})\leq A\bigg(\bigcup_{\substack{J\in\DD\\J\subsetneq I}}Q_J\bigg)=\frac{1}{2} A(Q_I).
$$
Generally, for every $m\geq1$, we note that
$$L_{I,m}=\bigcup_{\substack{J\in \mathcal{E}^{i+m}\\ J\subseteq I}}Q_J
=\bigcup_{\substack{J'\in\mathcal{E}^{i+m-1}\\ J'\subseteq I}}
  \bigcup_{\substack{J\in\mathcal{E}^{i+m}\\J\subseteq J'}}Q_J
=\bigcup_{\substack{J'\in\mathcal{E}^{i+m-1}\\ J'\subseteq I}}L_{J',1},$$
and argue by induction to obtain 
\begin{align*}
A(L_{I,m})&=\sum_{\substack{J'\in\mathcal{E}^{i+m-1}\\J'\subseteq I}}A(L_{J',1})\\
&\leq\frac{1}{2}\sum_{\substack{J'\in\mathcal{E}^{i+m-1}\\J'\subseteq I}}A(Q_{J'})\\
&=\frac{1}{2} A(L_{I,m-1})\\
&\leq\frac{1}{2^m}A(Q_I).
\end{align*}
Combining this with Lemma \ref{esi}, we establish that
\begin{align*}
\frac{\sigma(L_{I,m})}{\sigma(Q_I)}
\leq Cc_{\sigma}
  \left(\frac{A(L_{I,m})}{A(Q_I)}\right)^{\frac{1}{1+9[\sigma]_{B_{\infty}}}}
\leq Cc_{\sigma}
  2^{-\frac{m}{1+9[\sigma]_{B_{\infty}}}}
\end{align*}
for some absolute constant $C\geq1$. Choose
$$m_0:=\min\left\{m\in\mathbb{N}\colon
m\geq\frac{\log(4Cc_{\sigma})}{\log2}\left(1+9[\sigma]_{B_{\infty}}\right)\right\},$$
and then
\begin{equation}\label{1/4}
\frac{\sigma(L_{I,m_0})}{\sigma(Q_I)}
\leq Cc_{\sigma}
  2^{-\frac{m_0}{1+9[\sigma]_{B_{\infty}}}}\leq\frac{1}{4}.
\end{equation}

Let $E_I:=Q_I\setminus L_{I,m_0}=Q_I\setminus\bigcup_{\substack{J\in\mathcal{E}^{i+m_0}\\J\subseteq I}}Q_J$. We now show that $E_I$ is the subset we are seeking. On one hand,
\begin{align*}
\sum_{I\in\mathcal{E}}\chi_{E_I}\leq m_0
  \leq1+\frac{\log(4Cc_{\sigma})}{\log2}\left(1+9[\sigma]_{B_{\infty}}\right)
\lesssim(1+\log c_{\sigma})[\sigma]_{B_{\infty}}.
\end{align*}
On the other hand, for $I\in\mathcal{E}^i$, applying \eqref{fj} and \eqref{1/4}, we obtain 
\begin{align*}
\int_{Q_I}|f|\sigma \,dA
&=\frac{1}{4}\langle|f|\rangle_{\sigma,Q_I}\sigma(Q_I)+\frac{3}{4}\int_{Q_I}|f|\sigma \,dA\\
&\leq2^{-j-2}\sigma(Q_I)+\frac{3}{4}\sum_{\substack{J\in\mathcal{E}^{i+m_0}\\J\subseteq I}}\int_{Q_J}|f|\sigma \,dA
  +\frac{3}{4}\int_{E_I}|f|\sigma \,dA\\
&\leq2^{-j-2}\sigma(Q_I)+3\cdot2^{-j-2}\sum_{\substack{J\in\mathcal{E}^{i+m_0}\\J\subseteq I}}\sigma(Q_J)
  +\frac{3}{4}\int_{E_I}|f|\sigma \,dA\\
&=2^{-j-2}\sigma(Q_I)+3\cdot2^{-j-2}\sigma(L_{I,m_0})+\frac{3}{4}\int_{E_I}|f|\sigma \,dA\\
&\leq\frac{7}{8}\cdot2^{-j-1}\sigma(Q_I)+\frac{3}{4}\int_{E_I}|f|\sigma \,dA\\
&\leq\frac{7}{8}\int_{Q_I}|f|\sigma \,dA+\frac{3}{4}\int_{E_I}|f|\sigma \,dA,
\end{align*}
which implies that
$$
\int_{Q_I}|f|\sigma \,dA\leq6\int_{E_I}|f|\sigma \,dA.
$$
The proof is complete.
\end{proof}

The following lemma can be found in \cite{CRr2020}*{Lemma 3.1}.

\begin{lemma}\label{sum-esti}
If $\gamma_1,\gamma_2,\eta>0$ and $\delta\geq0$, then
$$\sum_{k,j\geq0}\min\left\{\gamma_12^{-k},\eta\gamma_22^{-j}2^{(\delta-1)k}\right\}
\leq C_{\delta}\gamma_1\log_2(\mathrm{e}+\gamma_2)+\frac{\eta}{2}.$$
\end{lemma}

We also need the weak-type $(1,1)$ bound for the weighted dyadic maximal operator. Recall that, given a dyadic grid $\DD$ and a weight $\sigma$, the weighted maximal operator $M^{\DD}_{\sigma}$ is given by
$$M^{\DD}_{\sigma}f:=\sup_{I\in\DD}\langle |f|\rangle_{\sigma,Q_I}\chi_{Q_I}$$
and that 
\begin{equation}\label{M-weak}
\|M_{\sigma}^{\DD}\|_{L^1(\mathbb{H},\sigma)\to L^{1,\infty}(\mathbb{H},\sigma)}\leq1.
\end{equation}

We are now ready to prove Theorem \ref{MixedWeak}.

\begin{proof}[Proof of Theorem \ref{MixedWeak}]
By Proposition \ref{SparseBound}, it suffices to establish the desired estimate for $S_{\DD}$ for a dyadic grid $\DD$. Moreover, by monotone convergence, we need only consider finite subcollections of $\DD$ (so that later in the proof we can use Lemma~\ref{E_I}). By abuse of notation, we will again denote such a finite subcollection by $\DD$. Fix $\lambda>0$ and a nonnegative function $f\in L^1(\mathbb{H},uv)$ with $\|f\|_{L^1(\mathbb{H},uv)}=1$. The bound \eqref{M-weak} implies that
\begin{equation}\label{>1/2}
uv\left(\left\{z\in\mathbb{H}\colon M^{\DD}_{uv}f(z)>\frac{\lambda}{2}\right\}\right)
\leq\frac{2}{\lambda}.
\end{equation}

{\bf Case 1:} $u\in B_1$ and $v\in B_p(u)$ for $1<p<\infty$. Let
$$G_{\lambda}:=\left\{z\in\mathbb{H}\colon\frac{S_{\DD}(fv)(z)}{v(z)}>\lambda \,\,\text{and}\,\, M^{\DD}_{uv}f(z)\leq\frac{\lambda}{2}\right\}.$$
Since $v<\lambda^{-1}S_{\DD}(fv)=S_{\DD}(\lambda^{-1}fv)$ on $G_{\lambda}$, we have
\begin{align}\label{G-sum}
uv(G_{\lambda})&\leq\int_{\mathbb{H}}S_{\DD}(\lambda^{-1}fv)u\chi_{G_{\lambda}}\,dA\\
&=\sum_{I\in\DD}\langle\lambda^{-1}fv\rangle_{Q_I}\int_{Q_I}u\chi_{G_{\lambda}}\,dA\nonumber\\
&=\sum_{I\in\DD}\langle u\rangle_{Q_I}\langle\chi_{G_{\lambda}}\rangle_{u,Q_I}
  \int_{Q_I}\lambda^{-1}fv\,dA\nonumber\\
&\leq[u]_{B_{1}}\sum_{I\in\DD}\langle\lambda^{-1}f\rangle_{uv,Q_I}
  \langle\chi_{G_{\lambda}}\rangle_{u,Q_I}uv(Q_I).\nonumber
\end{align}
By the definition of the set $G_{\lambda}$, we know that the last summation in \eqref{G-sum} only concerns the intervals $I\in\DD$ with $\langle\lambda^{-1}f\rangle_{uv,Q_I}\leq1/2$, and hence we can decompose the summation as follows. For $k,j\geq0$, let $\mathcal{E}_{k,j}$ be the collection of $I\in\DD$ such that
$$2^{-j-1}<\langle\lambda^{-1}f\rangle_{uv,Q_I}\leq2^{-j}\quad \mathrm{and}\quad
2^{-k-1}<\langle\chi_{G_{\lambda}}\rangle_{u,Q_I}\leq2^{-k},$$
and let
$$
s_{k,j}:=\sum_{I\in\mathcal{E}_{k,j}}\langle\lambda^{-1}f\rangle_{uv,Q_I}
  \langle\chi_{G_{\lambda}}\rangle_{u,Q_I}uv(Q_I).$$
Then by \eqref{G-sum},
\begin{equation}\label{G-sum-kj}
uv(G_{\lambda})\leq[u]_{B_{1}}\sum_{k,j\geq0}s_{k,j}.
\end{equation}

We now claim that for any $k,j\geq0$,
\begin{align}\label{min}
s_{k,j}\leq\min\Big\{ &{\lambda^{-1}C2^{-k}(1+\log c_{uv})[uv]_{B_{\infty}},}\\
&\qquad{2^{p+1}2^{-j}2^{k(p-1)}[uv]_{B_{\infty}}[v]_{B_{p}(u)}uv(G_{\lambda}) } \Big\}.\nonumber
\end{align}
For the first estimate, we apply Lemma \ref{E_I} to the function $\lambda^{-1}f$ and the weight $uv$ to get subsets $E_I\subseteq Q_I$ for each $I\in \mathcal{E}_{k,j}$ such that
$$\int_{Q_I}\lambda^{-1}fuv\,dA\leq 6\int_{E_I}\lambda^{-1}fuv\,dA$$
and
$$\sum_{I\in\mathcal{E}_{k,j}}\chi_{E_I}\leq C(1+\log c_{uv})[uv]_{B_{\infty}}.$$
Consequently,
\begin{align*}
s_{k,j}&\leq2^{-k}\sum_{I\in\mathcal{E}_{k,j}}\int_{Q_I}\lambda^{-1}fuv\,dA\\
&\leq6\cdot2^{-k}\sum_{I\in\mathcal{E}_{k,j}}\int_{E_I}\lambda^{-1}fuv\,dA\\
&\leq C2^{-k}(1+\log c_{uv})[uv]_{B_{\infty}}\int_{\mathbb{H}}\lambda^{-1}fuv\,dA\\
&=\lambda^{-1}C2^{-k}(1+\log c_{uv})[uv]_{B_{\infty}}.
\end{align*}
For the second estimate in \eqref{min}, we use Lemma \ref{cup} to deduce that
\begin{align*}
s_{k,j}\leq2^{-j-k}\sum_{I\in\mathcal{E}_{k,j}}uv(Q_I)
\leq 2\cdot2^{-j-k}[uv]_{B_{\infty}}uv\left(\bigcup_{I\in\mathcal{E}_{k,j}}Q_I\right).
\end{align*}
It follows from Lemma \ref{weight-change} that for any $I\in\mathcal{E}_{k,j}$,
\begin{align*}
2^{-k-1}&<\langle\chi_{G_{\lambda}}\rangle_{u,Q_I}
  \leq[v]^{\frac{1}{p}}_{B_{p}(u)}\langle\chi_{G_{\lambda}}\rangle_{uv,Q_I}^{\frac{1}{p}}
\leq[v]^{\frac{1}{p}}_{B_{p}(u)}\left(M_{uv}^{\DD}\chi_{G_{\lambda}}(z)\right)^{\frac{1}{p}}
\end{align*}
for any $z \in Q_I$, which, combined with \eqref{M-weak}, implies that
\begin{align*}
s_{k,j}
&\leq 2\cdot2^{-j-k}[uv]_{B_{\infty}}uv\left(\bigcup_{I\in\mathcal{E}_{k,j}}Q_I\right)\\
&\leq 2\cdot2^{-j-k}[uv]_{B_{\infty}}
  uv\left(\left\{M^{\DD}_{uv}\chi_{G_{\lambda}}>[v]^{-1}_{B_{p}(u)}2^{-(k+1)p}\right\}\right)\\
&\leq 2^{p+1}2^{-j}2^{k(p-1)}[uv]_{B_{\infty}}[v]_{B_{p}(u)}uv(G_{\lambda}).
\end{align*}
Therefore, the inequality \eqref{min} holds. 

Combining \eqref{G-sum-kj} with \eqref{min} gives that
\begin{align*}
uv(G_{\lambda})\leq\sum_{k,j\geq0}
  \min\Big\{ &{\lambda^{-1}C2^{-k}(1+\log c_{uv})[uv]_{B_{\infty}}[u]_{B_{1}},}\\
& {2^{p+1}2^{-j}2^{k(p-1)}[uv]_{B_{\infty}}[v]_{B_{p}(u)}[u]_{B_{1}}uv(G_{\lambda})} \Big\}.
\end{align*}
We now apply Lemma \ref{sum-esti} with $\gamma_1=\lambda^{-1}C(1+\log c_{uv})[uv]_{B_{\infty}}[u]_{B_{1}}$, $\gamma_2=2^{p+1}[uv]_{B_{\infty}}[v]_{B_{p}(u)}[u]_{B_{1}}$, $\delta=p$ and $\eta=uv(G_{\lambda})$ to establish that
\begin{align*}
uv(G_{\lambda})
&\leq\lambda^{-1}C_{p}(\log c_{uv}+1)[uv]_{B_{\infty}}[u]_{B_{1}}
    \log_2\left(\mathrm{e}+2^{p+1}[uv]_{B_{\infty}}[v]_{B_{p}(u)}[u]_{B_{1}}\right)
    +\frac{1}{2}uv(G_{\lambda})\\
&\leq\lambda^{-1}C_{p}(\log c_{uv}+1)[uv]_{B_{\infty}}[u]_{B_{1}}
    \left(\log([uv]_{B_{\infty}}[v]_{B_{p}(u)}[u]_{B_{1}})+1\right)
    +\frac{1}{2}uv(G_{\lambda}).
\end{align*}
Consequently,
$$uv(G_{\lambda})\lesssim \lambda^{-1}(\log c_{uv}+1)[uv]_{B_{\infty}}[u]_{B_{1}}
  \left(\log([uv]_{B_{\infty}}[v]_{B_{p}(u)}[u]_{B_{1}})+1\right),$$
which, in conjunction with \eqref{>1/2}, implies that
\begin{align*}
uv\left(\left\{v^{-1}S_{\DD}(fv)>\lambda\right\}\right)
 \lesssim \frac{1}{\lambda}(\log c_{uv}+1)[uv]_{B_{\infty}}[u]_{B_{1}}
\left(\log([uv]_{B_{\infty}}[v]_{B_{p}(u)}[u]_{B_{1}})+1\right).
\end{align*}

{\bf Case 2:} $v\in B_1$ and $u\in B_1(v)$. Since $u\in B_{1}(v)$, Lemma \ref{weight-change} implies that
$$M^{\DD}_{v}f\leq[u]_{B_{1}(v)}M_{uv}^{\DD}f.$$
This, together with \eqref{M-weak}, gives that
$$
uv\left(\left\{M^{\DD}_{v}f>\frac{\lambda}{2}\right\}\right)
\leq uv\left(\left\{M^{\DD}_{uv}f>\frac{\lambda}{2[u]_{B_{1}(v)}}\right\}\right)
\leq\frac{2[u]_{B_{1}(v)}}{\lambda}.
$$
Let
$$
G_{\lambda}:=\left\{z\in\mathbb{H}\colon\frac{S_{\DD}(fv)(z)}{v(z)}>\lambda \,\,\text{and} \,\, M^{\DD}_{v}f(z)\leq\frac{\lambda}{2}\right\}.
$$
Then it is enough to show
\begin{equation}\label{enough}
uv(G_{\lambda})\lesssim
\frac{1}{\lambda}(\log c_v+1)[v]_{B_1}[v]_{B_{\infty}}[u]_{B_{1}(v)}
\big(\log([v]_{B_{1}}[uv]_{B_{\infty}})+1\big).
\end{equation}
Since $v<\lambda^{-1}S_{\DD}(fv)=S_{\DD}(\lambda^{-1}fv)$ on $G_{\lambda}$, similarly as before, we have
\begin{align*}
uv(G_{\lambda})
&\leq\sum_{I\in\DD}\langle\lambda^{-1}fv\rangle_{Q_I}\int_{Q_I}u\chi_{G_{\lambda}}\,dA\\
&=\sum_{I\in\DD}\langle v\rangle_{Q_I}
  \langle\lambda^{-1}f\rangle_{v,Q_I}\int_{Q_I}u\chi_{G_{\lambda}}\,dA\\
&\leq[v]_{B_{1}}\sum_{I\in\DD}\langle\lambda^{-1}f\rangle_{v,Q_I}
  \langle\chi_{G_{\lambda}}\rangle_{uv,Q_I}uv(Q_I).
\end{align*}
For $k,j\geq0$, let $\mathcal{E}_{k,j}$ be the collection of $I\in\DD$ such that
$$
2^{-j-1}<\langle\lambda^{-1}f\rangle_{v,Q_I}\leq2^{-j}
\quad\mathrm{and}\quad
2^{-k-1}<\langle\chi_{G_{\lambda}}\rangle_{uv,Q_I}\leq2^{-k},
$$
and let
$$s_{k,j}:=\sum_{I\in\mathcal{E}_{k,j}}\langle\lambda^{-1}f\rangle_{v,Q_I}
\langle\chi_{G_{\lambda}}\rangle_{uv,Q_I}uv(Q_I).$$

We claim that
\begin{align}\label{min2}
s_{k,j}\leq\min\Big\{\lambda^{-1}C2^{-k}(1+\log c_v)[v]_{B_{\infty}}[u]_{B_{1}(v)},
4\cdot2^{-j}[uv]_{B_{\infty}}uv(G_{\lambda})\Big\}.
\end{align}
For the first inequality, we apply Lemma \ref{E_I} with the function $\lambda^{-1}f$ and the weight $v$ to find a subset $E_I\subseteq Q_I$ for each $I\in\mathcal{E}_{k,j}$ such that
$$\int_{Q_I}\lambda^{-1}fv\,dA\leq 6\int_{E_I}\lambda^{-1}fv\,dA$$
and
$$\sum_{I\in\mathcal{E}_{k,j}}\chi_{E_I}\leq C(1+\log c_v)[v]_{B_{\infty}}.$$
Consequently,
\begin{align*}
s_{k,j}&\leq2^{-k}\sum_{I\in\mathcal{E}_{k,j}}\langle\lambda^{-1}f\rangle_{v,Q_I}uv(Q_I)\\
&\leq6\cdot2^{-k}\sum_{I\in\mathcal{E}_{k,j}}\langle u\rangle_{v,Q_I}\int_{E_I}\lambda^{-1}fv\,dA\\
&\leq6\cdot2^{-k}[u]_{B_{1}(v)}\sum_{I\in\mathcal{E}_{k,j}}\int_{E_I}\lambda^{-1}fuv\,dA\\
&\leq\lambda^{-1}C2^{-k}(1+\log c_v)[v]_{B_{\infty}}[u]_{B_{1}(v)}.
\end{align*}
For the second inequality in \eqref{min2}, we apply Lemma \ref{cup}, \eqref{M-weak} and the fact that $M^{\DD}_{uv}\chi_{G_{\lambda}}\geq\langle\chi_{G_{\lambda}}\rangle_{uv,Q_I}>2^{-k-1}$ on $Q_I$ for all $I\in\mathcal{E}_{k,j}$ to establish that
\begin{align*}
s_{k,j}&\leq2^{-k-j}\sum_{I\in\mathcal{E}_{k,j}}uv(Q_I)\\
&\leq 2\cdot2^{-k-j}[uv]_{B_{\infty}}uv\left(\bigcup_{I\in\mathcal{E}_{k,j}}Q_I\right)\\
&\leq 2\cdot2^{-k-j}[uv]_{B_{\infty}}
    uv\left(\left\{M^{\DD}_{uv}\chi_{G_{\lambda}}>2^{-k-1}\right\}\right)\\
&\leq 4\cdot2^{-j}[uv]_{B_{\infty}}uv(G_{\lambda}),
\end{align*}
which proves \eqref{min2}. Therefore,
\begin{align*}
uv(G_{\lambda})\leq\sum_{k,j\geq0}
  \min\Big\{ {\lambda^{-1}C2^{-k}(1+\log c_v)[v]_{B_{\infty}}[v]_{B_{1}}[u]_{B_{1}(v)}},
  {4\cdot2^{-j}[v]_{B_{1}}[uv]_{B_{\infty}}uv(G_{\lambda})}  \Big\}.
\end{align*}
Applying Lemma \ref{sum-esti} with $\gamma_1=\lambda^{-1}C(1+\log c_v)[v]_{B_{\infty}}[v]_{B_{1}}[u]_{B_{1}(v)}$, $\gamma_2=4[v]_{B_{1}}[uv]_{B_{\infty}}$, $\delta=1$, and $\eta=uv(G_{\lambda})$, we conclude \eqref{enough} and finish the proof.
\end{proof}


\section{Extensions to simple domains}\label{SimpleDomainsSection}

In this section, we describe how our bounds extend to general simple domains. The key feature of a simple domain $\Omega$ is that the geometry of its boundary $\partial \Omega$ is sufficiently well-behaved so that one can define an appropriate quasi-metric that turns $\partial \Omega$ into a space of homogeneous type with respect to the Lebesgue surface measure $S$. One can then endow the boundary with a dyadic structure using the results of Hyt\"onen and Kairema from \cite{HK2012} and, in turn, construct a dyadic structure on $\Omega$. More specifically, in the case of a strongly pseudoconvex domain $\Omega$, one equips $\partial \Omega$ with the horizontal subspace metric (also called the Carnot-Carath\'{e}odory distance) as described in \cite{BB2000}. A similar process can be carried out in the setting of finite type domains in $\C^2$ using the sub-Riemannian geometry of the boundary; see \cite{HWW20202}. Although the boundary geometry can be more structurally complicated in the remaining two cases of simple domains, the quasi-metric can be defined using the scaling approach in \cites{M1989, M1991, MConvex}. Since these details have been carried out for general simple domains in other manuscripts, such as \cite{HWW20202}, for example, we will only present the main ideas in the case of strongly pseudoconvex domains and simply refer to the above references for the process on other simple domains.

\subsection{Extensions to strongly pseudoconvex domains}
Here and henceforth, $\rho$ will denote a fixed, smooth defining function for $\Omega=\{z \in \C^n: \rho(z)<0\}.$ Given a point $\zeta \in \partial \Omega$, denote the horizontal subspace of the (real) tangent space at $\zeta$ by
$$
    H_\zeta := \{Z \in \C^n\colon \langle \overline{\partial}\rho(\zeta), Z \rangle=0\},
$$
where $\overline{\partial}\rho(\zeta)$ is the vector with $j^{\text{th}}$ component equal to $\frac{\partial\rho}{\partial \overline{z_j}}(\zeta)$ and $\langle \cdot \, , \cdot \, \rangle$ is the standard inner product on $\C^n.$ We define the horizontal metric for $z,w \in \partial \Omega$ by
\begin{align*}d(z,w) := \inf\bigg \{ &\int_{0}^1 |\alpha'(t)| \, dt\colon \text{ $\alpha \colon  [0,1] \rightarrow \partial \Omega$ is piecewise smooth} \\
& \quad\text{ with $\alpha(0)=z$, $\alpha(1)=w$, and $\alpha'(t) \in H_{\alpha(t)}$ for all $t \in [0,1]$} \bigg \}. \end{align*}
Let $B(z, \delta):=\{w \in \partial \Omega\colon d(z,w)<\delta\}$ be a metric ball on the boundary. Let $S$ denote the Lebesgue surface measure on $\partial \Omega$ (recall $\Omega$ is smoothly bounded). Since $(\partial \Omega, d, S)$ is a geometrically doubling space of homogeneous type, we can equip it with a dyadic structure using \cite{HK2012}*{Theorem 2.2}. The following proposition makes this precise.

\begin{prop} [\cite{HWW20202}, Lemma 3.6] \label{prop: BoundaryKubes}
Let $\delta>0$ be sufficiently small and let $s>1.$ There exist points $\{p_j^k\}_{\substack{j \in \mathcal{P}_k, \\ k \in \N }} \subseteq \partial\Omega$, where $\mathcal{P}_k:=\{0,1,2, \cdots, N_k\}$, and an associated collection of subsets $\mathcal{Q}=\{Q_j^k\}$ of $\partial \Omega$ such that $p_j^k \in Q_j^k$ and the subsets $Q_j^k$ satisfy 
\begin{enumerate}
    \item for each $k$, $\{p_j^k\}$ is a maximal set of points on $\partial \Omega$ satisfying $d(p_j^k, p_i^k)> s^{-k} \delta$ for all $i$ and $j$, i.e., if $p \in \partial \Omega$ does not belong to the collection $\{p_j^k\}$, then there exists an index $j_0$ so that $d(p, p_{j_0}^k) \leq s^{-k} \delta$;

    \item for each fixed $k$, $\{Q_j^k\}_{j \in \N}$ forms a partition of $\partial \Omega$;

    \item for any $k<\ell$ and any $i,j $, either $Q_j^k \supseteq Q_i^\ell$ or $Q_j^k \cap Q_i^\ell= \emptyset$;

    \item there exist positive constants $c$ and $C$ such that 

    $$B(p_j^k, c s^{-k} \delta) \subseteq Q_j^k \subseteq B(p_j^k, C s^{-k} \delta)$$
    for all $j$ and $k$; and

    \item each $Q_j^k$ contains at most $N$ sets of the form $Q_i^{k+1}$, where $N$ is independent of $j$ and $k$. \looseness=-1

\end{enumerate}

\end{prop}

Moreover, these dyadic sets serve as good approximations for arbitrary boundary balls, so long as a sufficient number of dyadic systems are used. 

\begin{prop}[\cite{HWW20202}, Lemma 3.7] \label{DyadicBoundary} 
Let $\delta$ and $\{p_j^k\}$ be as in Proposition \ref{prop: BoundaryKubes}. There exist finitely many collections $\{ \mathcal{Q}_{\ell}\}_{\ell=1}^M$ such that 

\begin{enumerate}
    \item each collection $\mathcal{Q}_{\ell}$ is associated with collections of dyadic points $\{p_j^k\}$ and sets $\{Q_j^k\}$ with respect to which all the properties in Proposition \ref{prop: BoundaryKubes} are satisfied; and 

    \item if $z \in \partial \Omega$ and $r>0$ is sufficiently small, then there exist $Q_{j_1}^{k_1} \in \mathcal{Q}_{\ell_1}$ and $Q_{j_2}^{k_2} \in \mathcal{Q}_{\ell_2}$ such that
    $$ Q_{j_1}^{k_1} \subseteq B(z, r) \subseteq Q_{j_2}^{k_2} \quad \text{and} \quad S(Q_{j_1}^{k_1})  \approx S(B(z, r)) \approx S(Q_{j_2}^{k_2}).$$
\end{enumerate}

\end{prop}

Let $\varepsilon>0$ be sufficiently small and define $N_\varepsilon(\partial \Omega):= \{\zeta \in \Omega\colon \operatorname{dist}(\zeta, \partial \Omega)< \varepsilon\}$. For $z \in N_\varepsilon(\partial \Omega)$, we denote by $\pi(z)$ the normal projection of $z$ to the boundary $\partial \Omega$; note that the map $\pi$ is well-defined and smooth as long as $\varepsilon$ is chosen small enough. These boundary dyadic sets then induce dyadic ``tents'' inside $\Omega$ according to the following definition:

$$\widehat{K}_j^k:= \{z \in \Omega\colon \pi(z) \in Q_j^k \,\, \text{and} \,\, |\pi(z)-z|< s^{-2k} \delta^2\}.$$
The disjoint subsets of these tents, which we call ``kubes", are defined by  
$$K_j^k:= \widehat{K}_j^k \setminus \bigcup_{\ell: \,p_\ell^{k+1} \in Q_j^k} \widehat{K}_{\ell}^{k+1}.$$
Here, we need to choose $\delta$ small enough so that $\widehat{K}_j^k$ always lies in $N_\varepsilon(\partial \Omega)$.
For notation, let $\widehat{K}_{0}^{-1}:=\Omega$, $K_0^{-1}:= \Omega \setminus \bigcup_{\substack{k \in \N \\ j \in \mathcal{P}_k}}\widehat{K}_j^k$ and $\mathcal{P}_{-1}:=\{0\}$. The tents and kubes then satisfy the following properties.

\begin{lemma}[\cite{HWW20202}, Lemma 3.9 and Lemma 3.11] \label{TentsandKubes} 
If $\mathcal{T}=\{ \widehat{K}_j^k\}$ is a system of tents induced by $\mathcal{Q}$ as given in Proposition \ref{prop: BoundaryKubes}, then 
\begin{enumerate}
    \item for any $\widehat{K}_j^k, \widehat{K}_i^{k+1} \in \mathcal{T}$, either $\widehat{K}_j^k \supseteq \widehat{K}_i^{k+1}$ or $\widehat{K}_j^k \cap \widehat{K}_i^{k+1}= \emptyset$; 

\item the kubes $K_j^k$ are pairwise disjoint and $\bigcup_{j,k} K_j^k=\Omega$; and 

\item for any $j$ and $k$, we have 
$$
    V(\widehat{K}_j^k) \approx V(K_j^k) \approx s^{-2kn} \delta^{2n},
$$ 
where $V$ denotes the Lebesgue volume measure on $\mathbb{C}^n$.
\end{enumerate}  
\end{lemma}
\noindent We note that the tents $\widehat{K}_j^k$ and kubes $K_j^k$ exactly correspond to the Carleson tents $Q_I$ and their pairwise disjoint subsets $T_I$ in $\mathbb{H}$ for $I \subseteq \mathcal{D}$, respectively. 

The relevance of this construction of tents lies in the fact that this geometry is well-adapted to the Bergman kernel $K_{\Omega}$ of a bounded, strongly pseudoconvex domain with smooth boundary $\Omega\subseteq\mathbb{C}^n$. In particular, we have the following dyadic domination of $K_{\Omega}$.
\begin{lemma}[\cite{HWW20202}, Theorem 4.6]\label{SimpleKernelBound}
If $\Omega \subseteq \mathbb{C}^n$ is a bounded, strongly pseudoconvex domain with smooth boundary, then 
$$|K_{\Omega}(z,w)|  \lesssim   V(\Omega)^{-1} \chi_{\Omega \times \Omega}(z,w)+ \sum_{\ell=1}^M \sum_{j,k} V(\widehat{K}_j^k)^{-1} \chi_{\widehat{K}_j^k \times \widehat{K}_j^k}(z,w)$$
for all $z,w \in \Omega$.
\end{lemma}

Using the integral representation of the positive Bergman operator 
$$
    P^+f(z) = \int_{\Omega}|K_{\Omega}(z,w)||f(w)|\,dV(w),
$$
Lemma \ref{SimpleKernelBound} immediately leads to a dyadic domination of $P^+$ as in Proposition \ref{SparseBound}, and hence the proofs of all our bounds naturally adapt to these domains. The only nuance is that the entire domain $\Omega$ must be included as a Carleson tent in the Muckenhoupt basis, since in this case $\Omega$ has finite measure. In particular, the (dyadic) $B_p$ condition is defined by 
$$
    [\sigma]_{B_p}:= \sup_{\substack{k \in \N \cup \{-1\} \\ j \in \mathcal{P}_k}} \langle \sigma \rangle_{\widehat{K}_j^k}\langle \sigma^{1-p'}\rangle_{\widehat{K}_j^k}^{p-1}.
$$

\begin{rem}
We remark that this $B_p$ constant can be defined in a slightly different way, which detects the quantitative differences between averages taken over the whole domain $\Omega$ versus the tents which are strict subsets; see \cite{HWW20202}*{Theorem 1.2}. The adjusted characteristic does give slightly more precise estimates, but, for the sake of simplicity, we do not pursue this refinement here.
\end{rem}


\section{Acknowledgments}

We thank \'Oscar Ram\'irez for conversations regarding the optimality of Theorem \ref{WeightedWeakLp}.


\end{document}